\newcommand{\baray}{\begin{array}{rcl}}
\newcommand{\earay}{\end{array}}
\newcommand{\barray}{\begin{array}{rcl}}
\newcommand{\earray}{\end{array}}

\newcommand{\mL}{\mathfrak{L}}
\newcommand{\mX}{\mathfrak{X}}
\newcommand{\mZ}{\mathfrak{Z}}
\newcommand{\mC}{\mathfrak{C}}
\newcommand{\mG}{\mathfrak{G}}
\newcommand{\mF}{\mathfrak{F}}
\newcommand{\mT}{\mathfrak{T}}
\newcommand{\mS}{\mathfrak{S}}

\newcommand{\levy}{L\'evy }

\newcommand{\bcase}{\begin{cases}}
\newcommand{\ecase}{\end{cases}}

\newcommand\cadlag{c{\'a}dl{\'a}g }

                       \newcommand\del[1]{}
\newcounter{ggg}

\def\Re{\mathop{ \rm Re }}

\newcommand{\n}{|}

\newcommand{\lk}{\left}
\newcommand{\lqq}{\lefteqn}
\newcommand{\rk}{\right}

\newcommand{\LL}{{\rm I \kern -0.2em L}}

\newcommand{\be} {\begin{enumerate} }
\newcommand{\ee} {\end{enumerate} }

\newcommand{\CO}{{{ \mathcal O }}}

\newcommand{\FG}{{{\mathfrak{G} }}}

\newcommand{\CZ}{{{ \mathcal Z }}}
\newcommand{\CT}{{{S}}}

\newcommand{\CB}{{{ \mathcal B }}}

\newcommand{\CM}{{{ \mathcal M }}}

\newcommand{\BF}{{{ \mathbb{F} }}}
\newcommand{\CF}{{{ \mathcal F }}}

\newcommand{\CL}{{{ \mathcal L }}}

\newcommand{\RR}{{\mathbb{R}}}
\newcommand{\Rb}[1]{{\mathbb{R}_{#1}}}

\newcommand{\DD}{\mathbb{D}}

\newcommand{\NN}{\mathbb{N}} %\newcommand{\OWi}{\mbox{ otherwise }}

\newcommand{\PP}{{\mathbb{P}}}

\newcommand{\EE}{ \mathbb{E} }
\newcommand{\TT}{{\rm I \kern -0.2em T}}

\newcommand{\DEQS}{\begin{eqnarray*}}
\newcommand{\EEQS}{\end{eqnarray*}}
\newcommand{\DEQSZ}{\begin{eqnarray}}
\newcommand{\EEQSZ}{\end{eqnarray}}
\newcommand{\DEQ}{\begin{eqnarray}}
\newcommand{\EEQ}{\end{eqnarray}}

\documentclass[Manuscriptpreprint,12pt]{elsarticle}
\usepackage{amsmath}   % many want amsmath extensions
\usepackage{amsfonts,amssymb, color}
\usepackage{latexsym}
\input colordvi

\usepackage{amsfonts}
\begin{document}

\begin{frontmatter}



\title{Global  solutions to
stochastic Volterra  equations driven by L\'evy noise\footnote{This work was supported by the FWF-Project P17273-N12 and Marsden Fund project number UOO1418.}
}


%\title{}


\author{Erika Hausenblas}
\address{Department of Mathematics, Montanuniversitaet Leoben, 8700 Leoben,
 Austria}
\ead{erika.hausenblas@unileoben.ac.at}

\author{Mih\'aly Kov\'acs}
\address{Department of Mathematical Sciences, Chalmers University of Technology and University of Gothenburg, Sweden}
\ead{mihaly@chalmers.se}

%\ead{mkovacs@maths.otago.ac.nz}
%\subjclass

%\pagestyle{myheadings} \markright{\today}\pagenumbering{arabic}

%\theoremstyle{plain}

%
%

\begin{abstract}
In this paper we investigate the existence and uniqueness of semilinear stochastic Volterra equations driven by multiplicative L\'evy noise of pure jump type. In particular, we consider the
equation
\begin{equation*}
\hspace{-0.5cm}
\left\{ \begin{aligned} du(t) & = \lk( A\int_0 ^t b(t-s) u(s)\,ds\rk) \, dt  + F(t,u(t))\,dt \\
& {} + \int_ZG(t,u(t), z) \tilde \eta(dz,dt)  +
\int_{Z_L}G_L(t,u(t), z) \eta_L(dz,dt) ;\,  t\in (0,T],\\
u(0)&=u_0. \end{aligned} \right.
\end{equation*}
 where $Z$ and $Z_L$ are Banach spaces,  $\tilde \eta$ is a time-homogeneous compensated Poisson random measure on $Z$ with intensity measure $\nu$ (capturing the small jumps), and
 $\eta_L$ is a time-homogeneous
Poisson random measure  on $Z_L$ independent to $\tilde{\eta}$ with finite intensity measure $\nu_L$ (capturing the large jumps).
 Here, $A$ is a selfadjoint operator on a Hilbert space $H$, $b$ is a scalar memory  function and $F$, $G$ and $G_L$ are nonlinear mappings. We provide conditions on $b$, $F$
 $G$ and $G_L$ under which a unique global solution exists. We also present an example from the theory of linear viscoelasticity where our result is applicable. \end{abstract}

%\maketitle


%q\end{frontmatter}

% \linenumbers


\end{frontmatter}

%\runauth{E. Hausenblas, M. Kov\'acs}

%\end{opening}

%{\footnotesize\tableofcontents}

%\title
\newtheorem{proof}{Proof}[section]
\newtheorem{theorem}{Theorem}[section]
\newtheorem{notation}{Notation}[section]
\newtheorem{claim}{Claim}[section]
\newtheorem{lemma}[theorem]{Lemma}%[section]
\newtheorem{corollary}[theorem]{Corollary}%[section]
\newtheorem{example}[theorem]{Example}%[section]
\newtheorem{assumption}[theorem]{Assumption}
\newtheorem{tlemma}{Technical Lemma}[section]
\newtheorem{definition}[theorem]{Definition}%[section]
\newtheorem{remark}[theorem]{Remark}%[section]
\newtheorem{hypo}{H}
\newtheorem{proposition}[theorem]{Proposition}%[section]
\newtheorem{Notation}{Notation}
\renewcommand{\theNotation}{}
\newtheorem{ass}[theorem]{Assumption}

\renewcommand{\labelenumi}{\alph{enumi}.)}

% \tableofcontents

\section{Introduction}\label{sec_intro}

In this work we analyse the existence and uniqueness
of a class of stochastic Volterra equations driven by L\'evy noise.
To be more precise, let $Z$ and $Z_L$ be two Banach   spaces,  $\tilde \eta$ be a compensated Poisson random measure on $Z$ with intensity measure $\nu$, and $\eta_L$ be a
Poisson random measure on $Z_L$ independent to $\tilde{\eta}$ with finite intensity measure $\nu_L$.
%The first Poisson random measure is of infinite activity capturing the small jumps,
%the second one is of finite activity capturing the large jumps.
With given mappings $G, G_L$ and $F$ and initial data $u_0$, satisfying certain conditions specified later,
%such that $\int_Z \int_{\RR ^d } |c(z(x))|\, dx\nu(dz)<\infty$.
we are interested in the solution of the equation
\begin{equation}\label{SDE-intro2} %
\hspace{-0.5cm}
\left\{ \begin{aligned} du(t) & = \lk( A\int_0 ^t b(t-s) u(s)\,ds\rk) \, dt  + F(t,u(t))\,dt \\
& {} + \int_ZG(t,u(t), z) \tilde \eta(dz,dt)  +
\int_{Z_L}G_L(t,u(t), z) \eta_L(dz,dt) ;\,  t\in (0,T],\\
u(0)&=u_0. \end{aligned} \right.
\end{equation}
%
%
%In this paper we investigate under which conditions on the spaces $Z,Z_L$, the memory function $b$ and on the mappings   $G$, $G_L$  and $F$ there exists a mild solution to
%\eqref{SDE-intro2}.
The examples that we have in mind that can be described via \eqref{SDE-intro2} are linear models of viscoelastic materials perturbed by random abrupt forcing. Therefore, we impose conditions on the memory kernel $b$ that are typical in that setting, see \cite[Chapter 5]{pruss} for more details.   We also would like to mention that Volterra equations, especially with the specific kernel $b(t)=c_{\rho}t^{\rho-2}$, $1<\rho<2$, are an important modeling tool in many fields of sciences, such as mechanical engineering (including viscoelasticity) \cite{atan,mainardi2,mainardi,mainardi1}, chemistry \cite{chemical}, physics (heat conduction) \cite{yuri3,yuri1}, neurology \cite{neuro1,neuro} and several other fields \cite{frac0,frac}.\\

Gaussian perturbations are not always appropriate for interpreting real data in a
reasonable way. This is the case when, for example, one wants to model abrupt pulses by a random perturbation process or when the time scale of the random process is much finer then the time scale of the deterministic process or when {extreme events} occur relatively frequently. In all these cases, a natural mathematical modeling framework could be based on L\'evy processes or general semimartingales with jumps.

%Here, in this work we are interested in the existence and uniqueness
%of the stochastic Volterra  equation with a \levy process of pure jump type,
%i.e.\ in the following equation %We are interested in the solution of the following equation
%%
%%
%\begin{equation}\label{SDE-introL}\hspace{1.0cm}
%\left\{ \begin{aligned} du(t) & = \lk( A\int_0 ^t b(t-s) u(s)\,ds\rk) \, dt  + f(t,u(t))\,dt \\
%& {} + g(t,u(t)) dL(t) ;\,  t\in (0,T],\\
%u(0)&=u_0. \end{aligned} \right.
%\end{equation}
%%
%%
%Here $L$ denotes a $H$--valued \levy process.
%In this article we investigate under which conditions on the space $Z$ and on the mappings   $g$, $g_L$  and $f$ there exists a strong  solution to \eqref{SDE-intro}.
%
%
%When the stochastic perturbation is a Wiener process, the equation is well treated and existence and uniqueness of the solution is shown.
Existence and uniqueness of solutions for Gaussian noise driven Volterra equations in various frameworks have been treated by several authors, see, for example, \cite{baeumeretall,Barbu,bona1,BoDP,BoFa,bona2,CDaPP,DLO,Karczewskabook,Karczewska2012,Karczewska2009,Keck} for an incomplete list of papers and the references therein. For an additive fractional Brownian motion driven linear parabolic stochastic Volterra equation we refer to \cite{Sperlich} while for linear additive square-integrable local martingale driven  stochastic Volterra equations we refer \cite{SchV}. Finally, we mention \cite{DLO1} where the regularity (but not pathwise) of a class of linear stochastic Volterra equations were investigated in an additive Poisson random measure setting but with different smoothing properties of the deterministic solution operator than the one in the present paper (as subdiffusion, rather than dissipative waves in viscoelastic materials as in the present paper, was considered). As far as we know nonlinear stochastic Volterra equations with multiplicative non-Gaussian noise have not been analysed in the literature.

The main results of the paper are Theorem \ref{local_ex} and Theorem \ref{global}. In Theorem \ref{local_ex} we prove, using a fixed point argument, that \eqref{SDE-intro2} has a unique mild solution with certain \cadlag path regularity when $G_L=0$ and $G$ satisfy some integrability conditions with respect to $\nu$. In Theorem \ref{global} we show that a global mild solution with the same \cadlag path regularity to \eqref{SDE-intro2} still exists when $G_L \neq 0$ and the finite measure $\nu_L$ satisfy a mild condition at infinity.  %However, given a L\'evy process, one can construct easily the corresponding Poisson random measure.

\medskip

The paper is organized as follows. First, in Section \ref{deterministic} we introduce some necessary preliminary material and also state our main assumption on the memory kernel
and on the operator $A$. Then, in Section \ref{sec:bddmom}, we
investigate the existence and uniqueness of a mild solution of \eqref{SDE-intro2} with $G_L=0$ and $G$ satisfying some integrability conditions with respect to $\nu$. In Section \ref{sec:ubddj} we extend this result by allowing $G_L\neq 0$ but no integrability condition is assumed on $G_L$ as we take $\nu_L$ to be a finite measure.
Here we will use the representation of the L\'evy process or compound Poisson process associated with $\eta_L$ in terms of a sum over all its jumps.
If there would be no memory term, one would simple glue together the solutions between the large jumps. However, if there is a non-trivial memory term, the solution
is not a Markov process, hence, it is not possible to apply the gluing technique.  We deal with this difficulty by introducing special Hilbert spaces, taking into account
the jump times and sizes. In Section \ref{reaction1-s-t} we present an example describing the velocity field of a synchronous viscoelastic material in the presence of an abrupt force field modeled by space-time L\'evy noise.

Finally, we note that we use Poisson random measures to describe the stochastic perturbation for two reasons.
Firstly, using a notation based on Poisson random measures simplify many calculations and,
secondly, the use of Poisson random measures allow for a more general framework than L\'evy jump processes. For a precise connection between Poisson random measures and L\'evy processes and integration with respect to these we refer to \cite{dettweiler}, \cite[Appendix]{mishi5}, \cite[Chapters 6-8]{PESZAT+ZABZCYK} and \cite[Chapter 4]{sato}.

\del{
\begin{notation}
We denote by $B_{p,q}^ s(\RR^d)$ the natural Besov spaces defined in \cite[Chapter 1.2.5, p.\ 8]{triebelII} or \cite[Definition 2, p.\ 8]{runst}.
%$F_{p,q}^ s(\RR^d)$,
For $k\in\NN_0$ we denote by $H^k_p(\RR^d)$ the classical Sobolev spaces defined in \cite[Chapter 3, Definition 3.1]{triebelharoske}.
For $\delta\ge 0$ and $p\in(0,\infty]$ let $L^p_\delta(\RR^d)=\{ v\in L^p(\RR^d )$ with $\int_{\RR^d } (1+|x|^{2})^\frac \delta2\, |u(x)|^p\, dx<\infty\}$.

For any index $p\in[1,\infty]$ we denote throughout the paper the conjugate element by $p'$. In particular, we have $\tfrac 1p + \tfrac 1 {p'}=1$.

\end{notation}
}

\section{Preliminaries}\label{deterministic}

In this section we shortly introduce some  preliminary results and state our main assumptions on the the operator $A$, memory kernel $b$.
The operator $A$ is typically an elliptic differential operator.
We  will assume the following.
\begin{ass}\label{ass1}
The operator $A:D(A)\to H$ is an unbounded, densely defined, linear, self adjoint, negative-definite operator with compact inverse.
\end{ass}
%Under Assumption \ref{ass1} it is well-known that $A$ generates a bounded analytic $C_0$--semigroup $(T(t))_{t\geq 0}$ on $H$ (see, for example \cite[Example 3.7.5]{MR2798103}).
Throughout the paper we will use the fractional powers of the operator $-A$ and the spaces associated with them. Let $\alpha \in\RR$.
Then, for $\alpha <0$ we define $H^{A}_{\alpha }$ to be
the closure of $H$ under the norm $\n x\n_{H^{A}_{\alpha }} := \lk|
(-A)^{\alpha } x \rk|_H$. If $\alpha  \ge 0$, then
$H^{A}_{\alpha }=D((-A)^{\alpha })$ and $\n
x\n_{H^{A}_{\alpha }} := \lk| (-A)^{\alpha } x \rk|_H$. For more details, we refer to, for example
\cite[Chapter 2]{Pazy:83}.

%
%The kernel $b$ in the memory term is a so-called creep function.
%As mentioned in the introduction, $b$ will be a so called creep functions.
%\begin{definition}
%A function $b:(0,\infty)\to \RR$ is called a creep function if $b$ is nonnegative, nondecreasing, and concave. A creep function has the standard form
%$$
%b(t) = b_0+b_\infty t + \int_0 ^t b_1(s)\, ds,\quad t>0,
%$$
%where $b_0=b(0+)\ge 0$, $b_\infty=\lim_{t\to\infty} b(t)/t=\inf_{t>0} b(t)/t\ge 0$, and $b_1(t)=\cdot{b}(t)-b_\infty$ is nonnegative,
%nonincreasing, $\lim_{t\to\infty} b_1(t)=0$.
%\end{definition}
%The kernel $b$ is called $k$--regular, if there exists a $C>0$ such that its Laplace transform $\hat b$ satisfies
%$$
%|\lambda |^ j |\hat b^{(j)}(\lambda))|\le C |\hat b(\lambda)|
%$$
%for all $\Re(\lambda)>0$ and all $0\le j\le k$. In addition we call $b$
% sectorial of angle less than $\frac \pi 2$, if $$\sup\{|\arg(\hat b(\lambda)|,\Re\lambda >0\}< \frac \pi 2.$$
%Finally, the kernel $b$ is called $k$--monotone $(k\ge 2)$, if $b$ is $k$--$2$--times continuously differentiable
%on $(0,\infty)$, $(-1)^ nb^{(n)}(t)\ge 0$ for $t\ge 0$ and $0\le n\le k-2$ and $(-1)^ {k-2} b^{(k-2)}$ is nondecreasing and convex.

\medskip

Next we formulate our assumptions on the memory term typical in the theory of viscoelasticity
so that the deterministic equation exhibits a parabolic behaviour, c.f. \cite[Assumption 1]{mishi1} and \cite[Hypothesis (b)]{CDaPP}.
The kernel $b$ is called $k$\textit{-monotone} ($k\ge 2$) if $b$ is $(k-2)$-time continuously differentiable on $(0,\infty)$, $(-1)^nb^{(n)}(t)\ge 0$ for $t>0$ and $0\le n\le k-2$ and $(-1)^{k-2}b^{(k-2)}$ is nonincreasing and convex, see \cite[Definition 3.4]{pruss}. In particular, throughout the paper we assume the following.
\begin{ass}\label{ass2}
The kernel $0\neq b\in L^1_{loc}(\mathbb{R}_+)$ is $4$-monotone and $$\lim_{t\to \infty}b(t)=0.$$ Furthermore,
\begin{equation}\label{eq:sector}
\rho  := 1 + \frac{2}{\pi}\sup \{ | \mathrm{arg} \, \widehat b(\lambda) |, \; \Re\lambda >0 \} \in (1,2).
\end{equation}
\end{ass}
It follows from \cite[Proposition 3.10]{pruss} that for 3-monotone and locally integrable kernels $b$, condition \eqref{eq:sector} is equivalent to
%(see \cite[Prop. 3.3, Chap. I]{pruss}).
\begin{equation*}\label{eq:b-smooth}
\lim_{t\rightarrow 0} \frac{\frac 1t\int_0^t s b(s) \, ds}{\int_0^t -s \dot{b}(s) \, ds} < +\infty.
\end{equation*}
%Sometimes we will also need the following extra assumption.
%\begin{ass}\label{ass3}
% The Laplace transform $\widehat{b}$ of the kernel $b$ satisfies $$\widehat{b}(\lambda)\leq C\lambda^{\rho-1} $$ for $\lambda \gg1$.
%\end{ass}
\begin{remark}\label{rem25}
A simple, but important, example of a kernel $b$ that satisfies Assumption \ref{ass2} is
$$
b(t) = \Gamma(\rho-1)^ {-1} t^{\rho-2} e^ {-\eta t},
$$
with $1<\rho<2$ and $\eta\ge 0$.

%A more general class of kernels arise in the linear theory of viscoelasticity. A  kernel $b$ is called $k$\textit{-monotone} ($k\ge 2$) if $b$ is $(k-2)$-time continuously differentiable on $(0,\infty)$, $(-1)^nb^{(n)}(t)\ge 0$ for $t>0$ and $0\le n\le k-2$ and $(-1)^{k-2}b^{(k-2)}$ is nonincreasing and convex, see \cite[Definition 3.4]{pruss}. If $b$ is $4$-monotone, $\lim_{t\to \infty}b(t)=0$, and
%\begin{equation*}
%\limsup_{t\rightarrow 0,\infty} \frac{\frac 1t\int_0^t s b(s) \, \dd s}{\int_0^t -s \dot{b}(s) \, \dd s} < +\infty,
%\end{equation*}
%then Assumption \ref{ass2} is satisfied  with
%\begin{equation*}
%\rho  := 1 + \frac{2}{\pi}\sup \{ | \mathrm{arg} \, \hat b(\lambda) |, \; \Re\lambda >0 \} \in (1,2),
%\end{equation*}
%see, \cite{baeumeretall}.
%\end{remark}
%\begin{remark}
%It follows from \cite[estimate (3.6)]{MP97} that the sector condition \eqref{eq:sector} implies that $\widehat{b}(\lambda)\geq C\lambda^{\rho-1} $ for $\lambda  \gg1$. Therefore, if both Assumptions \ref{ass2} and \ref{ass3} are satisfied then $\widehat{b}(\lambda) \sim \lambda^{\rho-1} $ for $\lambda \gg  1$.
\end{remark}

Under Assumptions \ref{ass1} and \ref{ass2} on $A$ and $b$ it follows that there exists a strongly continuous family $(S(t))_{t\ge 0}$ of resolvents  such that the function
$u(t)=S(t)u_0$, $u_0\in H$, is the unique solution of the Cauchy problem
\DEQS
u(t) = A \int_0 ^t B(t-s) u(s)\, ds + u_0,\quad t\ge 0,
\EEQS
with $B(t)=\int_0 ^t b(s)\, ds$, see, \cite[Corollary 1.2]{pruss}.
Observe that the family of resolvents $(S(t))_{t\ge 0}$ does not have the semigroup property because of the presence of the memory term. However, the family of
resolvents $(S(t))_{t\ge 0}$ satisfy certain smoothing properties as shown in \cite[Lemma A.4]{baeumeretall}, see also, \cite[Proposition 2.5 and Remark 2.6]{mishi1}.
\begin{lemma}\label{smoothing}
Suppose the operator $A$ satisfies Assumption \ref{ass1} and the convolution kernel $b$ satisfy Assumptions \ref{ass2}. Then
\begin{itemize}
\item[(a)]
$\| S(t)\| _{L(H,H_\alpha  ^A)} \le C t ^ {-\alpha\rho}$, $ t>0$, $\alpha\in [0,\frac 1 \rho]$;
\item[(b)]
$\|  \dot S(t)\|_{L(H,H_\alpha  ^A)} \le C t ^ {-\alpha\rho-1}$, $ t>0$, $\alpha\in [0,\frac 1 \rho]$.
%\item[(c)]
%$\| \dot S(t)\| _{L(H,H_{-\alpha}  ^A)}\le C\, \|b\|^ \alpha _{L^ 1 (0,t)\, }t ^ {\alpha-1}$, $ t>0$, $\alpha\in [0,1]$.
\end{itemize}
%If, in addition, $b$ satisfies Assumption \ref{ass3}, then (c) improves to
%$$\|\dot S(t)\| _{L(H,H_{-\alpha}  ^A)} \le C\, t ^ {\rho \alpha-1},\quad  t>0,\,\, \alpha\in [0,1].
%$$
\end{lemma}
%

%\begin{remark}
%Suppose the operator $A$ satisfies assumption \ref{ass1} and
% the convolution kernel $b$ satisfy Assumptions \ref{ass2}. Then, Proposition \ref{smoothing} reads as follows
%\begin{enumerate}
%\item
%$\| S(t)\| _{L(H,H_\alpha  ^A)} \le t ^ {-\alpha\rho}$, $ t>0$, $\alpha\in [0,\frac 1 \rho]$;
%\item
%$\|  \dot S(t)\|_{L(H,H_\alpha  ^A)} \le t ^ {-\alpha\rho-1}$, $ t>0$, $\alpha\in [0,\frac 1 \rho]$;
%\item
%$\| \dot S(t)\| _{L(H,H_{-\alpha}  ^A)}\le C\, \|b\|^ \alpha _{L^ 1 (0,t)\, }t ^ {\alpha-1}$, $ t>0$, $\alpha\in [0,1]$;
%\end{enumerate}
%\end{remark}

In the remainder of the section we will investigate deterministic and stochastic convolutions.
%\paragraph{The deterministic convolution term}
In order to proceed we first introduce some notation.
Let $E$ be a Banach space and let $T>0$ be fixed. %an interval.
%For a Banach space $E$ l
 For $1\le q<\infty$, and all $\lambda\ge 0$
 let $L^ q_\lambda ( 0,T ;E ) $ denote the
space of all integrable functions $\xi:I\to E$   such that
$$
\| \xi\|_{L^q_\lambda (0,T;E)}^q :=\int_0 ^T   e^ {-\lambda t}  \lk| \xi(t)\rk|^ q_{E}\, \,
dt<\infty.
$$

\subsection{Deterministic  convolutions}

\noindent
Let us denote the convolution of an appropriate function $\xi:[0,T]\to H$ and family of resolvents $(S(t))_{t\ge 0}$ by
\DEQSZ\label{det-con}
(\mC \xi)  (t) := \int_0^t \CT(t-r) \xi (r)\, dr,\quad t\in[0,T].
\EEQSZ
We have the following result.
\begin{proposition}\label{det_s}
Under the Assumptions \ref{ass1} and \ref{ass2},
for any $
0<\alpha<\frac 1\rho$, the operator $$\mC:L^q_\lambda ( 0,T ;H_{-\alpha} ^A)\to L^q_\lambda ( 0,T;H) $$
 is a bounded linear operator.
In particular, there exists a constant $C=C(\alpha,\lambda)>0$, with $$\lim_{ \lambda\to\infty} C(\lambda,\alpha)=0,$$ such that
for any $\xi\in L^ {q}_\lambda ( 0,T;H_{-\alpha} ^A)$ we have
$$
\| \mC \xi\|_{L^q_\lambda (0,T;H)} \le C\,
\|  \xi\|_{L^q_\lambda ( 0,T ;H_{-\alpha} ^A)}.
$$
\end{proposition}
\begin{proof}
{\rm Proposition \ref{det_s} can be shown by straightforward calculations using the Minkowski inequality and the Young inequality
for convolutions. Compare with \cite[Proof of Theorem 3.3]{baeumeretall}.
}
\end{proof}
%

%
%\paragraph{The stochastic convolution }\label{stochastic}
%
%

\subsection{Stochastic  convolutions}

In the next paragraph we would like to show a similar result for stochastic convolutions. First,
we introduce the setting we will
use throughout the whole paper. Let
$\mathfrak{A}=(\Omega,\CF,\BF,\PP)$ be a complete filtered probability space with right continuous filtration
 $\{\CF_t\}_{t\ge
0}$, denoted by  by $\BF$.
The random perturbation we are considering will be specified via a compensated Poisson random measure.
%
%
%
%A symmetric $\sigma$--finite
%measure $\nu$ on a separable Banach space $Z$ endowed with the Borel $\sigma$--algebra $\mathcal{B}(Z)$ is called a \textit{L\'evy measure} if  $\nu(\{0\})=0$ and the function
%$$a\mapsto \exp\lk( \int_E (\cos\la x,a\ra -1)\, \nu(dx)\rk)
%$$
%is the characteristic function of a Radon measure. A general $\sigma$--finite
%measure $\nu$ is called a L\'evy measure if the symmetric measure $\nu +\nu^-$ is a L\'evy measure (see \cite[p. 69]{linde}), where $\nu^-(B)=\nu(-B)$, $B\in\mathcal{B}(Z)$.
%
%In case, $Z$ is of $R$-type $p$,
%then the \levy measure satisfies for some $q\in[p,2]$ $\int_Z (|z|^q \wedge 1)
%\nu(dz)<\infty$ (see \cite[p. 75]{linde}). Observe, each Hilbert space if of $R$--type $2$, $L^ p$-spaces, with $p\in [1,2]$ are of $R$--type-$p$. Let $\lambda$ denote the Lebesgue measure on $[0,\infty)$.
%
%
\begin{definition}\label{def-Prm}
Let $Z$  be a separable Banach space and $\nu$ be a $\sigma$-finite measure on $(Z,\mathcal{B}(Z))$. \del{ and let $(\Omega,\CF,\BF,\PP)$ be a %complete filtered
probability space with filtration $\BF=\{\CF_t\}_{t\ge 0}$.} A (time-homogeneous)
{\sl Poisson random measure} $\eta$ on $(Z\times \RR_+,\mathcal{B}(Z)\otimes \mathcal{B}({\mathbb{R}_+}))$   over
$(\Omega,\CF,\BF,\PP)$ (Poisson random measure on $Z$ for short) with compensator $\gamma:=\nu\otimes \lambda$ is a family $\eta:=\{\eta(\omega,\cdot):\omega \in \Omega\}$ of nonnegative measures $\eta(\omega,\cdot)$
on $(Z\times \RR_+,\mathcal{B}(Z)\otimes \mathcal{B}({\mathbb{R}_+}))$ such that
{\begin{trivlist}
 \item[(i)]the mapping $\omega\to\eta(\omega,\cdot)$ is
 measurable $(\Omega,\CF)\to (M_I(Z\times \RR_+),\CM_I(Z\times \RR_+))$, where $M_I(Z\times \RR_+)$ is the set of nonnegative measures on $(Z\times \RR_+,\mathcal{B}(Z)\otimes
 \mathcal{B}({\mathbb{R}_+}))$ endowed with the $\sigma$--algebra $\CM_I(Z\times \RR_+)$ generated by all mappings $$M_I(Z\times \RR_+)\ni \rho\to \rho(\Gamma),\quad \Gamma \in
 \mathcal{B}(Z)\otimes \mathcal{B}({\mathbb{R}_+});$$
\item[(ii)] for each $B\in  \mathcal{B}(Z) \otimes
\mathcal{B}({\mathbb{R}_+}) $,
 $\eta(B):=\eta(\cdot,B): \Omega\to \bar{\mathbb{N}} $ is a Poisson random variable with parameter $\gamma(B)$;
\item[(iii)] $\eta$ is independently scattered, i.e. if the sets $
B_j \in   \CZ\otimes \mathcal{B}({\mathbb{R}_+})$, $j=1,\cdots,
n$, are  disjoint,   then the random variables $\eta(B_j)$,
$j=1,\cdots,n $, are independent.
%%\item[(iii)] for all $B\in  \CS $ and $I\in \mathcal{B}({\mathbb{R}_+})$, $\mathbb{E}\big[\eta (B\times I)\big]=\lambda(I)\nu(B)$, where $\lambda$ is the Lebesgue measure;
%\item[(iii)] for each $U\in \CZ$, the $\bar{\mathbb{N}}$-valued
%process $(N(t,U))_{t\ge 0}$  defined by
%$$N(t,U):= \eta(U \times (0,t]), \;\; t\ge 0$$
%is $\BF$-adapted and its increments are independent of the past,
%i.e.\ if $t>s\geq 0$, then $N(t,U)-N(s,U)=\eta(U \times (s,t])$ is
%independent of $\mathcal{F}_s$.
\end{trivlist}
}
The  difference between a time homogeneous  Poisson random measure
$\eta$  and its compensator $\gamma$, i.e.  $\tilde
\eta=\eta-\gamma$, is called a  {\em compensated Poisson random
measure}. The measure $\nu$ is called {\sl intensity measure} of $\eta$.
\end{definition}

%\begin{definition}
%The compensator of a random measure $\eta$ on a Banach space $Z$
%is the unique predictable measure $\gamma : \bcal ({\mathbb{R}}^0_{+})\times \CB(Z)  \to {\mathbb{R}}$, such that for any $A\in\CB(Z)$ the process
%$$
%\RR_+^0 \ni t \mapsto \eta(A\times [0,t])-\gamma([0,t]\times A)
%$$
%is a martingale over $\mathfrak{A}$.
%We will denote by $\tilde{\eta }$ the \it  compensated Poisson random measure \rm  defined by
%$\tilde{\eta }: = \eta - \gamma $.
%\end{definition}
%
%%
%\begin{remark}
%Assume that $\eta$ is a time homogeneous Poisson random measure
%with intensity $\nu$ on  $(Z,\CZ)$ over $(\Omega,\CF,\BF,\PP)$. It
%turns out that the compensator $\gamma$ of $\eta$ is uniquely
%determined and moreover
%$$
%\gamma: \CZ \times \CB(\RR^+)\ni (A,I)\mapsto  \nu(A)\times
%\lambda(I).
%$$
%Here $\lambda$ denotes the Lebesgue measure on $\RR$.
%The  difference between a time homogeneous  Poisson random measure
%$\eta$  and its compensator $\gamma$, i.e.  $\tilde
%\eta=\eta-\gamma$, is called a  {\em compensated Poisson random
%measure}. The measure $\nu$ is called {\sl intensity measure} of $\eta$.
%\end{remark}
%
\del{\begin{ass}\label{asslevy}
Fix $p\in(1,2]$ and $Z$ be a Banach space with Levy measure $\nu$. Throughout the paper we assume that the L\'evy measure $\nu$ will satisfy
$$
\int_{Z}1\wedge |z| ^p\,\nu(dz)<\infty.
$$
\end{ass}
}
Since it would
exceed the scope of the paper, we do not tackle here stochastic integration with respect to Poisson random measures.
A short summary of stochastic integration
is given
e.g.\  or Bre{\'z}niak and Hausenblas \cite{maxreg} or in \cite{PESZAT+ZABZCYK},
where
the stochastic integral is defined for all progressive measurable processes, i.e.\ for all processes which can be approximated by simple predictable c\'agl\'ag processes.
The following inequality corresponds to the It\^o isometry, compare e.g.\ \cite{dirksen}.
Assume that $E$ is a Hilbert space.
Then, for any $p\in [1,2]$ %, $p\le q<\infty$, %where $p$ is fixed %in Assumption \ref{asslevy}
and $q=p^n$ for some $n\in\NN$,
 there exists a constant $C>0$ such that
for all   progressively measurable processes
$$\xi:\Omega\times[0,T] \to L^p(Z,\nu ;E)$$
we have (see e.g.\ \cite{dirksen,levy2})
\DEQSZ\label{DD}
\\
\nonumber
\lqq{ \EE \lk| \int_0^ t \int _Z \xi(s,z)\,\tilde \eta(dz,ds)\rk|_E^ q } &&\\ &\le& C\lk\{ \EE \int_0^ T \int _Z  \lk|\xi(s,z)\rk|_E ^ q \nu(dz)\, ds
+\EE \lk( \int_0^ T \int _Z  \lk|\xi(s,z)\rk| _E^ p \nu(dz)\, ds\rk) ^ \frac qp
\rk\}.
\nonumber
\EEQSZ
By interpolation techniques one can prove inequality \eqref{DD} for all $q$ with $p\le q<\infty$.
We finish with the following version of the Stochastic Fubini Theorem (see \cite{zhu}).

% and with the definition of L\'evy measures.

\begin{theorem}\label{stochasticfubini}
Assume that $E$ is a Hilbert space  and
$$\xi:\Omega\times[a,b]\times \Rb{+}\times Z \to  E$$
is a progressively measurable process. Then, for each $T\in [a,b]$, we have a.s.
 \DEQS  \int_a^b\big[ \int_0^T \int_Z
\xi(s,r,z)\,\tilde\eta( dz,dr)\big] \,ds &=& \int_0 ^T \int_Z \big[
\int_a^b \xi(s,r,z)\,ds \big]\, \tilde \eta( dz,dr) . \EEQS
\end{theorem}

\medskip

We continue by introducing some notation.
As before, fix $1\le q<\infty$,  $0<m<\infty$, and a Banach space $E$. Then,
 $\CM^{m,q}_\lambda  ( 0,T ;E  )$ denotes the
space of all progressively measure processes $\xi:\Omega\times [0,T]\to E$  such that
$$
\|  \xi\|^m_{\CM^{m,q}_\lambda (0,T;E)}=\EE\lk( \int_0 ^T   e^ {-\lambda t}  \lk| \xi(t)\rk|^ q_{E}\, \,
dt\rk)^\frac mq<\infty.
$$
With this notation Proposition \ref{det_s} immediately implies the following result.
\begin{corollary}\label{det_st}
Under the Assumptions \ref{ass1} and \ref{ass2},
for any $0<\alpha<\frac 1\rho$, the operator $$\mC:\CM^{m,q}_\lambda (0,T;H_{-\alpha} ^A)\to \CM^{m,q}_\lambda (0,T;H) $$
 is a bounded linear operator.
In particular, there exists a constant $C=C(\alpha,\lambda)>0$, with $$\lim_{ \lambda\to\infty} C(\lambda,\alpha)=0,$$ such that
for any $\xi\in \CM^ {m,q}_\lambda (0,T;H_{-\alpha} ^A)$ we have
$$
\| \mC \xi\|_{\CM^{m,q}_\lambda (0,T;H)} \le C\, %(\lambda,\alpha ) \,
\|  \xi\|_{\CM^{m,q}_\lambda (0,T;H_{-\alpha} ^A)}.
$$
\end{corollary}

Next, we  investigate the properties of the
stochastic convolution operator defined, for an appropriate process $\xi:\Omega\times [0,T]\times Z\to E$, by  the formula,
\DEQSZ
\label{eqn-stoch_ter} \FG (\xi) &=&\left\{
\Rb{+}\ni t\mapsto \int_0 ^t\int_Z \CT(t-s) \xi(s,z)\;
\tilde\eta(dz,ds)\right\}.
\EEQSZ
For any $q,p\ge 1$ and $m>0$ let $\CM^{m,q} _\lambda (  0,T;L ^p(Z;\nu,E))  $ be the
space of all progressively measure processes $\xi:\Omega\times [0,T]\to  L^p(Z,\nu ;E)$ such that
$$
\lk\|  \xi  \rk\|^m_{\CM^{m,q}_\lambda (0,T;L^ {p}(Z,\nu;E)) }:=\EE\lk( \int_0 ^T   e^{-\lambda t}  \lk(  \int_Z \lk| \xi(t,z) \rk|^ p_{E} \,\nu(dz) \,\rk) ^{\frac qp } \,
dt\rk) ^{\frac mq }<\infty.
$$

\begin{proposition}\label{stconv}
Let $\eta$ be a Poisson random measure with intensity measure $\nu$.
Under the Assumptions \ref{ass1} and \ref{ass2},
for any  $\frac 1 {q\rho}>\alpha\ge 0$ and $\frac 1{\rho p}>\alpha_1\ge 0$
the stochastic convolution
$$\mathfrak{S}: \CM^ {q,q}_\lambda(0,T;L^ {q}(Z,\nu;H^A _{-\alpha}) ) \cap \CM ^{q,q}_{\lambda}(0,T;L^ {p}(Z,\nu;H^A _{-\alpha_1}) )
\longrightarrow \CM_\lambda^ {q,q}(0,T; H   )
$$
is linear and bounded.
In particular, there exists a constant $C=C(\alpha,\lambda)>0$, with $$\lim_{ \lambda\to\infty} C(\lambda,\alpha)=0,$$ such that
for any $$\xi \in \CM^ {q,q}_\lambda (0,T;L ^q(Z,\nu;H_{-\alpha} ^A))\cap  \CM^ {q,q}_\lambda (0,T;L ^p(Z,\nu;H_{-\alpha_1} ^A))$$ we have
\DEQS
\lqq{
\lk| \mS \xi\rk|^q_{\CM^{q,q}_\lambda (0,T;H)}}
&&
\\
& \le& C\,\lk({1 \over \lambda ^{1-q\rho}}\lk|  \xi  \rk|_{\CM^{q,q}_\lambda (0,T;L^ {q}(Z,\nu;H^A _{-\alpha})) }+
{1\over \lambda ^{1-p\rho} }\lk|  \xi  \rk|_{\CM^{q,q}_\lambda (0,T;L^ {p}(Z,\nu;H^A _{-\alpha_1})) }\rk)
 .
\EEQS
\end{proposition}
\begin{proof}
{\rm %
Applying inequality \eqref{DD} we obtain
\DEQS
\lqq{ \lk| \mS \xi\rk|^{q}_{\CM^{q,q}_\lambda (0,T;H)}
=
 \EE \int_0^T e^ {-\lambda t}
\lk| \mS(\xi)(t)\rk| ^q_{H } \, dt }
&&\\
&\le & C\,  \EE \int_0^T e^ {-\lambda t}  \int_0^ t \int_{Z}  \lk|\CT(t-s) \xi(s,z)\rk| ^q_{H}\nu(dz)\, ds \,
dt
\\&&{}+C\,\EE \int_0^T     e^ {-\lambda t} \lk(  \int_0^ t    \int_{Z}  \lk|\CT(t-s)
\xi(s,z)\rk| ^{p }_{ H}\nu(dz) \, ds\rk) ^{q\over p }  \,
dt. \EEQS
Using the estimates in Lemma \ref{smoothing}
we get %for any $\alpha_1$ with $\alpha_1<\frac 1 {p\rho}$
\DEQS
\lqq{ \lk| \mS \xi\rk|^{q}_{\CM^{q,q}_\lambda (0,T;H)}} &&
\\
 &\le &\EE \int_0^T  e^ {-\lambda t}   \int_0^ t\int_{Z} (t-s) ^ {-q\rho\alpha } \, \lk|
\xi(s,z)\rk| ^q_{H_{-\alpha}^A }\, \nu(dz)\,ds \, dt
\\&&{}+C\EE\, \int_0^ T e^ {-\lambda t}    \lk(  \int_0^ t   \int_{Z} (t-s) ^{-{p }\rho\alpha _1 }  \lk| \xi(s,z)\rk| ^{p}_{H_{-\alpha_1} ^A  }\nu(dz)\, ds\rk)
^{q\over p}  \, dt\\
 &= & C\EE\int_0^T     \int_0^ t\int_{Z} (t-s) ^ {-q\rho\alpha }e^  {-\lambda (t-s)} \, \lk|
\xi(s,z)\rk| ^q_{H_{-\alpha}^A }e^ {-\lambda s }\nu(dz)\, ds \, dt
\\
\lqq{ {}+C\EE\, \int_0^ T     \lk(  \int_0^ t   \int_{Z} (t-s) ^{-{p }\rho\alpha _1} e^  {-\frac pq \lambda (t-s)} \lk| \xi(s,z)\rk| ^{p }_{H_{-{\alpha_1}} ^A  }e^ {-\frac pq \lambda s
}\nu(dz)\, ds\rk)
^{\frac qp }  \, dt.
} &&%
\EEQS
Finally,
it follows by Fubini's Theorem and
the Young inequality for convolutions, that
 \DEQS
\lqq{ \lk| \mS \xi\rk|^{q}_{\CM^{q,q}_\lambda (0,T;H)}} &&
\\
&\le &
{C_1 \over \lambda ^{1-q\rho\alpha }}
\, \EE \int_0^T    \int_Z  e ^ {-\lambda t}  \lk|
\xi(t,z)\rk|_{H_{-\alpha} ^A } ^q \nu(dz)\,  dt
\\
&&{}+
{C_2\over (\lambda ^{1-p\rho\alpha _I}) ^\frac qp }
\,\EE \lk(  \int_0^T    e ^ {- \lambda t} \lk( \int_Z  \lk|
\xi(t,z)\rk|_{H^A_{-\alpha_1}  } ^{p } \nu(dz)\rk) ^ \frac qp  dt\rk),
\EEQS
which gives the assertion.
	}\end{proof}

\begin{proposition}\label{prop:cadlag}
Let $E$ be a Hilbert space and
$\xi:\Omega\times[0,T] \to L^p(Z,\nu ;E)$ be a progressively measurable processes. If $b$ satisfies Assumption \ref{ass2}, then the process
 $
 t \mapsto   \FG (\xi)(t)
 $
 has a \cadlag modification in $E$.
\end{proposition}
\begin{proof}
{\rm The statement follows the same way as the proof of \cite[Proposition 3]{PZ}, see also, \cite{HS} for the original idea for using an unitary dilation.
 Note that, for any $\omega>0$, the mapping $t\to e^{-\omega |t|}S(|t|)$ is strongly continuous, positive definite, and $S(0)=I$. This follows from \cite[Proposition 6]{PZ} provided that the function $t\to r_{\omega,\mu}(t):=e^{-\omega |t|}s_\mu(|t|)$, $\mu>0$ is fixed, is positive definite where $s_\mu$ is defined by
 $$
 \dot{s}_\mu(t)+\mu (b*s_\mu)(t)=0,\quad s_\mu(0)=1.
 $$
It is easily seen, using in particular the sector condition \eqref{eq:sector} on $b$, that $s_\mu$ is non-negative, continuous and bounded by 1, cf. \cite[Corollary 1.2]{pruss}. We calculate the Fourier transform of  $r_{\omega,\mu}$ in terms of the Laplace transform $\widehat{s_\mu}$ of $s_\mu$ as
as
$$
\int_{-\infty}^\infty e^{-i\beta t} r_{\omega,\mu} (t)\,dt = 2\Re \widehat{s_\mu}(\omega+i\beta)= 2\Re \frac{1}{\omega+i\beta +\mu \hat{b}(\omega+i\beta)}\ge 0,
$$
 where the non-negativity of the last term follows from the sector condition \eqref{eq:sector} on $b$. Hence, by Bochner's theorem (more precisely, the distributional version, the Bochner-Schwarz theorem), it follows that $r_{\omega,\mu}$ is positive definite being continuous and the Fourier transform of a tempered non-negative measure.
 Hence $t\to e^{-\omega |t|}S(|t|)$ has a unitary dilation $\{U(t)\}_{t\in \mathbb{R}}$, a strongly continuous unitary group, on some Hilbert space $E_1 \hookleftarrow E$ with $e^{-\omega |t|}S(|t|)x=\Pi U(t)x$, $x\in H$, where $\Pi:E_1\to E$ is the orthogonal projection. Then
\begin{align*}
\FG (\xi)(t)&=e^{\omega t}\int_0 ^t\int_Z e^{-\omega s}e^{-\omega(t-s)}\CT(t-s) \xi(s,z)\;
\tilde\eta(dz,ds)\\
&=e^{\omega t}\int_0 ^t\int_Z e^{-\omega s}\Pi U(t-s) \xi(s,z)\;
\tilde\eta(dz,ds)\\
&=e^{\omega t}\Pi U(t)\int_0 ^t\int_Z e^{-\omega s} U(-s) \xi(s,z)\;
\tilde\eta(dz,ds).
\end{align*}
The statement then follows from the regularity of stochastic integrals, see, for example, \cite[Theorems 1 and 2, pp 181-182]{dinculeanu}.
}
\end{proof}

\section{Stochastic Volterra equations with integrable jumps}\label{sec:bddmom}
In this section we start to investigate the existence and uniqueness of a mild solution to the Volterra equation
\begin{equation}\label{eqn0}\tag{SDE}
\left\{ \begin{aligned} du(t) & = \lk( A\int_0 ^t b(t-s) u(s)\,ds\rk) \, dt  + F(t,u(t))\,dt
\\ &{}  +\int_Z  G(t,u(t),z)\,\tilde \eta(dz,dt)
\\ &{}+\int_{Z_L}  G_L(t,u(t),z)\,\eta_L(dz,dt);\quad t\in (0,T],\\
u(0)&=u_0. \end{aligned} \right.
\end{equation}
Here, $\tilde{\eta}$ is a compensated Poisson random measure on $Z$ with intensity measure $\nu$ and $\eta_L$ is Poisson random measure on $Z_L$ independent to $\tilde{\eta}$ with finite intensity measure $\nu_L$.

 In case the stochastic perturbation is Gaussian we expect that the solution process is continuous.
In case the stochastic perturbation is L\'evy, or more generally given by a Poisson random measure, we do not necessarily expect this, but rather
that the solution belongs to the space of Skorokhod functions.
If  $Y$ denotes a separable and complete metric space and $T>0$ then the
space $\mathbb{D}([0,T];Y)$ denotes the space of all right continuous
functions $x:[0,T]\to Y$ with left limits.
The space of continuous function is usually equipped with the
uniform topology. But, since $\mathbb{D}([0,T];Y)$ is complete but not
separable in the uniform topology, we equip $\mathbb{D}([0,T];Y)$ with the
Skorokhod topology in which $\mathbb{D}([0,T];Y)$ is both separable and
complete. For more information about Skorokhod space and topology
we refer to Billingsley's book \cite{billingsley1} or Ethier and
Kurtz \cite{ethier}.

\begin{definition}[Mild Solution]\label{def:ms}
We say that $u$ is a mild solution of equation \eqref{eqn0}  in $H_0$,
if $u$ belongs $\PP$-a.s.\ to $\DD (0,T;H_0)$, $u$ is adapted to $(\CF_t)_{t\ge 0}$ and
we have $\PP$-a.s.,
$$\int_{0}^{t}\lk|S(t-s)F(s,u(s))\rk|_{H_0}\,ds<\infty ,
$$
and
$$
\int_{0}^{t}\int_Z\lk| S(t-s)G(s,u(s),z)\rk|_{H_0}^ p \, \nu(dz)ds<\infty.
$$
Furthermore, we suppose that for all $z\in Z_L$, and $r\in[0,T]$, $\PP$-a.s.,
$$
\lk|S(t-r)G_L(r-,u(r-),z)\rk|_{H_0} <\infty,
$$
and
for all $t\in [0,T]$ it holds $\PP$-a.s. that
\begin{align}\label{varcons}
u(t)&=S(t)u_0+\int_{0}^{t}S(t-s)F(s,u(s))\,ds
\nonumber
\\
\nonumber
&\qquad +\int_{0}^{t}\int_Z S(t-s)G(s,u(s),z)\,\tilde \eta (dz,ds)
\\
&\qquad
+ \int_{0}^{t}\int_{Z_L} S(t-s)G_L(s,u(s),z)\, \eta_L (dz,ds). %, \quad \textrm{a.s.}
\end{align}
\end{definition}
First we will show existence and uniqueness of a mild solution when  $G_L=0$; that is, for the case where there are no big jumps,
and therefore we consider the equation
\begin{equation}\label{eqnohnebj}\tag{SDE1}
\left\{ \begin{aligned} du(t) & = \lk( A\int_0 ^t b(t-s) u(s)\,ds\rk) \, dt + F(t,u(t))\,dt
\\ &\qquad+\int_{Z}  G(t,u(t),z)\,\tilde \eta(dz,dt) ;\quad t\in [0,T],\\
u(0)&=u_0. \end{aligned} \right.
\end{equation}
Our assumptions on $F$, $G$ and the initial condition are as follows.
\begin{ass}\label{hypogeneral}
Fix $p<q$. %<\tfrac{p}{p-1}$. %Let us assume that there exists a $\delta>1$ such that $(\delta-1)/\delta>p$ and put $q=\delta p$.
We will assume that

\renewcommand{\theenumi}{(\roman{enumi})}
\renewcommand{\labelenumi}{(\roman{enumi})}
\begin{enumerate}
\item there exists a number $\alpha_G<\frac 1 {q\rho}$ such that $G$ considered as a mapping $G: [0,T]\times H \to L^ p (Z,\nu;H_{-\alpha_G}^ A)\cap L^ q
    (Z,\nu;H_{-\alpha_G}^ A)$ is
    continuous  and for all $t\in[0,T]$ it is Lipschitz continuous with Lipschitz constant $L_G$ in the second variable;

\item there exists a  number $ \alpha_F< \frac 1 {\rho}$ such that $F: [0,T]\times H \to H_{-\alpha_F}^ A$ is continuous and
for all $t\in[0,T]$ it is Lipschitz continuous with Lipschitz constant $L_F$ in the second variable;
\item there exists a  number $\alpha_I<\frac 1 {q\rho}$ such that  $u_0\in L^ q(\Omega; H_{-\alpha_I}^ A)$ is  $\CF_0$-measurable.
\end{enumerate}
\end{ass}
%
%\misi{We would need some additional assumptions on the $t$-dependence of $F$ and $G$ such as measurability or so.}\erika{yes}

\noindent
Now, we can formulate one of our main results.

\begin{theorem}\label{local_ex}
Suppose that Assumption \ref{ass1} and Assumption \ref{ass2} are satisfied.
Let $\eta$ be a Poisson random measure with intensity measure $\nu$. %and that the Poisson random measure $\eta$. % satisfies Assumption \ref{asslevy}.
In addition, suppose that the coefficients $F$ and $G$ and the initial condition $u_0$
and the numbers $\alpha_F,\alpha_G$ and $\alpha_I$ satisfy Assumption \ref{hypogeneral}.
%Then, there exists a unique mild solution $u\in \CM^{q,q} (0,T ;H  )$ of equation \eqref{eqnohnebj}.
If
\begin{equation*}
\alpha\ge \alpha_I,\quad \del{\left((\alpha_G-\alpha)+1\right)  \frac {pq}{q-p}< 1} \alpha\ge \alpha_G, \mbox{ and } (\alpha_F-\alpha)\rho<1-\frac{1}{q},
\end{equation*}
then there exists a  mild solution of \eqref{eqnohnebj} in $H_{-\alpha}$. In particular, % such that
$$\PP\lk( u\in \DD([0,T];H_{-\alpha}^ A)\rk)=1.$$
%\item[(b)] If $b$ also satisfies Assumption \ref{ass3}, and  $\alpha\in\RR$ satisfies $\alpha- \alpha_G\le 1$ and
%\begin{equation*}
%\alpha\ge \alpha_I,\quad \del{\left((\alpha_G-\alpha)\rho+1\right)  \frac {pq}{q-p}< 1}\alpha\ge \alpha_G, \mbox{ and } (\alpha_F-\alpha)\rho<1-\frac{1}{q},
%\end{equation*}
%then there exists a  mild solution of \eqref{eqnohnebj} in $H_{-\alpha}$. In particular, % such that
%$$\PP\lk( u\in \DD([0,T];H_{-\alpha}^ A)\rk)=1.$$
If, in addition,  also
\DEQSZ\label{addass}
(\alpha_F-\alpha_I)\rho<1-\frac 1q \quad \mbox{{and}}\quad \alpha_I\ge \alpha_G,
\EEQSZ
then, for each $t\in [0,T]$, we have  $u(t)\in L ^q(\Omega;H_{-\alpha_I} ^A)$.
%\misi{I don't see the point here unless $\alpha_I<0$ or if we want to say that $u\in L^{\infty}(0,T;L ^q(\Omega;H_{-\alpha_I} ^A))$. Part I of the theorem  shows that $u(t)\in L ^q(\Omega;H)$ for almost all $t$... }
%
\end{theorem}
\begin{proof}
{\rm First we will prove the existence and uniqueness of a process $u$ satisfying \eqref{varcons} using Banach's fixed point theorem.  Then, we will investigate the regularity of the trajectories of $u$ in $H_{-\alpha}$; that is,\ the \cadlag property, to show that $u$ a mild solution in $H_{-\alpha}$ according to Definition \ref{def:ms}. Finally, we will prove that $u(t)\in L ^q(\Omega;H_{-\alpha_I} ^A)$
whenever \eqref{addass} holds.
\medskip

\paragraph{\bf Step I}
Put
$$\mathfrak{X} _\lambda := \CM^{q,q}_\lambda (0,T ;H )   .
$$
For fixed $u_0\in L^ q(\Omega; H_{-\alpha_I}^ A)$ let us define the integral operator
\DEQSZ\label{mZ}
\\
\nonumber (\mathfrak{Z}u)(t):= (\mT u_0)(t)+
\lk(\mathfrak{F}u\rk) (t)+\lk( \mathfrak{G}u\rk)(t), \quad t\in[0,T],\, u\in\mX_\lambda,
  \EEQSZ
where
$\mT:H_{-\alpha_I}^ A \to \mathfrak{X}$
is defined by
$$(\mT u_0)(t) := \CT(t) u_0,
\quad  u_0\in  H_{-\alpha_I}^ A, $$
$\mF=\mC\circ F$, where $\mC$ is defined in \eqref{det-con}, and  $\mG=\mS\circ g$, where $\mS$ is defined in \eqref{eqn-stoch_ter}.
We will show that there exists a $\lambda>0$ such that for any  $u_0\in L^\rho(\Omega;H_{-\alpha_I}^ A)$,
 $\mZ$ maps $\mathfrak{X}_\lambda$ into $\mathfrak{X}_\lambda$, and
that $\mZ$ is a strict contraction on $\mathfrak{X}_\lambda$.

The fact that $\mT$ maps $H_{-\alpha_I}^ A $ in $ \mX_\lambda$,  follows from the Lemma \ref{smoothing}.
In particular, a straightforward calculations gives
$$
\lk\| \mT u_0\rk\|^ q_{L^ {q}(0,T;H)  } \le \int_0^T e^ {-\lambda t}|\CT(t) u_0|_H^ q \, dt\le
 C\int_0^T e^ {-\lambda t} t ^ {-\rho\alpha_Iq} \, dt\, \n u_0\n_{H_{-\alpha_I}^ A}^q,
$$
and therefore,
\DEQSZ\label{operatorCT}
 \lk\| \mT u_0\rk\|_{L^ {q}(0,T;H)}  &\le &
 C\Gamma( 1-\rho\alpha_Iq)^ \frac 1q \, \lambda ^ {\rho  \alpha_I-\frac 1q}\, \n u_0\n _{H_{-\alpha_I}^ A}.
\EEQSZ
Corollary \ref{det_st} and the  Lipschitz continuity of $F$ imply that the operator
\DEQSZ\label{operatorf}
 \mathfrak{F} :\mX_\lambda \longrightarrow \mX_\lambda.
\EEQSZ
is Lipschitz continuous  with Lipschitz constant $L_F\, C(\alpha_F,\lambda)$ such that
$\lim_{\lambda\to \infty }$ $  C(\alpha_F,\lambda)=0$.
Proposition \ref{stconv} and the  Lipschitz continuity of $G$ implies that the operator
\DEQSZ\label{operatorg}
 \mathfrak{G} :\mX_\lambda \longrightarrow \mX_\lambda.
\EEQSZ
is Lipschitz continuous too with Lipschitz constant $L_G\, C(\alpha_G,\lambda)$ where
$$\lim_{\lambda\to \infty }  C(\alpha_G,\lambda)=0.
$$
Hence, for all $u_0\in H_{-\alpha_I}^ A$ and for all
$u\in \mX_\lambda$ we have
$$
\mathfrak{Z}u=\mT u_0+ \mathfrak{F} u +\mathfrak{G} u\in
\mX_\lambda.
 $$
In particular, the operator $\mZ$ is Lipschitz continuous with Lipschitz constant $L_\lambda $, such that $L_\lambda\to 0$ as $\lambda \to \infty$.
Hence, for $\lambda $ sufficiently large, there exists a fixed point in $\mX_\lambda$.
%

%% % % % % % % % % % % % % % % % % % % % % % % % % % % % % % % % % % % % % % % % % % % % % % % % % % % % % % % % % % % %

\medskip

\paragraph{\bf Step II}
Next, we show that $u$ has a \cadlag in modification in $H_{-\alpha} ^A$.
Notice that the resolvent family $S$ is strongly continuous in $H_{-\alpha}^A$; that is, for all $x\in H_{-\alpha}^A$
the mapping $$[0,\infty)\ni t\mapsto \CT(t) x \in H_{-\alpha}^A
$$
is continuous. As $u_0\in H^A_{-\alpha}$ $\PP$-a.s.
we have that
\DEQSZ\label{ersteszz}
\PP\lk( \lim_{h\to 0} \lk|\lk[ \CT(t+h)-\CT(t)\rk] u_0\rk|_{H^A_{-\alpha}}=0\rk) =1.
\EEQSZ
%Here, we will construct a sequence of sets $\{A_n:n\in\NN\}\subset \CF$, $A_n\subset A_{n+1}$,  such that
%\begin{itemize}
%  \item $\lim_{n\to\infty} \PP\lk( A_n\rk)=1$,
%    \item $A_n \subseteq \lk\{   \lim_{h\to 0} \lk|\lk[ \CT(t+h)-\CT(t)\rk] u_0\rk|_{H^A_{-\alpha}} =0 \rk\}$.
%\end{itemize}
%Since $u_0$ is a $\CF_0$--measurable and $u_0\in L^ q(\Omega; H_{-\alpha_I}^ A)$, for  all $n\in\NN $ there exists a compact set $\bar K_n\subset H_{-\alpha_I}^ A$ such that
%$\PP\lk( u_0\not\in \bar K_n\rk) \le 2^{-n}$.
%Let $K_1=\bar K_1$ and $K_{n}=\bar K_{n}\cup K_{n-1}$.
%
%Fix $n\in\NN$.
%Since $K_n$ is compact, the entity
%$$
%(0,1)\ni h\mapsto \sup_{\phi\in K_n}\lk|( \CT(h)-\CI)\phi\rk|_{H^A_{-\alpha}}$$ converges  to zero for $h\to 0$.
%Let us define $$ A_n :=\lk\{ u_0\in K_n\rk\}.
%$$
%Then,
%$$
%\PP\lk( \lim_{h\to 0} \lk|( \CT(h)-\CI) u_0\rk|_{H^A_{-\alpha}}=0\rk) \ge  \PP\lk( u_0\in K_n\rk)=\PP(A_n), \quad n\in\NN,
%$$
%and therefore
%$$
%\PP\lk( \lim_{h\to 0}\lk|( \CT(h)-\CI) u_0\rk| _{H^A_{-\alpha}} =0 \rk) \ge \lim_{n\to\infty}\PP\lk( u_0\in K_n\rk) = \lim_{n\to\infty} (1-2^{-n})=1.
%$$
%Hence we have shown \eqref{ersteszz}.

Next we show that the process
$$[0,\infty)\ni t\mapsto  \mF u (t) =\int_0^ {t} \CT( t-s)  F(s,u (s)) \, ds
$$
has even a continuous modification in $H_{-\alpha}^A$. Let $0\le r<t \le T$ and $\sigma>1$ and $\delta>1$ arbitrary. Then
\begin{align*}
&\EE\left\| \mF u (t)- \mF u (r)\right\|_{H_{-\alpha}}^\delta
\leq C\left(\EE\left\|\int_r^ {t} \CT( t-s)  F(s,u (s)) \, ds\right\|_{H_{-\alpha}}^\delta\right.\\
&\left.\quad+\EE\left\|\int_0^ {r} (\CT( t-s) -\CT(r-s)) F(s,u (s)) \, ds\right\|_{H_{-\alpha}}^\delta\right):=(e_1)^\delta+(e_2)^{\delta}.
\end{align*}
To estimate for $e_1$ we use H\"older's inequality and, from Step I, the fact that $u\in \CM^{q,q}_\lambda (0,T ;H )$ together with the Lipschitz continuity of $F$ and Jensen's inequality, to calculate
\begin{align}
& e_1\leq \EE \int_r^{t} \|\CT(t-s)\|_{L(H_{-\alpha_F},H_{-\alpha})}\|F(s,u(s))\|_{H_{-\alpha_F}}\, ds\\
&\leq  C \EE \int_r^ {t} (t-s)^{\left(-(\alpha_F-\alpha)\rho\right)\wedge 0}(1+\|u(s)\|)\,ds\\
&\leq C  \left(\int_r^ {t} (t-s)^{(-(\alpha_F-\alpha)\rho)\frac{q}{q-1}\wedge 0}\,ds\right)^{\frac{q-1}{q}} \left(\int_0^r(1+\EE\|u(s)\|^q)\,ds\right)^{\frac{1}{q}} \\
&\leq C (t-r)^{\left(-(\alpha_F-\alpha)\rho+\frac{q-1}{q}\right) \wedge \frac{q-1}{q}},
\end{align}
provided that $(-(\alpha_F-\alpha)\rho)\frac{q}{q-1}>-1$, that is, $(\alpha_F-\alpha)\rho<1-\frac{1}{q}.$
Next we bound $e_2$. We use Fubini's Theorem, H\"older's inequality and, from Step I, the fact that $u\in \CM^{q,q}_\lambda (0,T ;H )$ together with the Lipschitz continuity of $F$ and Jensen's inequality to get
\begin{align*}
&e_2\leq \EE  \int_0^r \int_r^t \|\dot{\CT}(v-s)\|_{L(H_{-\alpha_F},H_{-\alpha})}\,dv \|F(s,u(s))\|_{H_{-\alpha_F}}\, ds\\
&\leq C\EE \int_r^t \int_0^r (v-s)^{\left(-(\alpha_F-\alpha)\rho-1\right)\wedge -1}(1+\|u(s)\|)\,ds\,dv\\
&\leq C \int_r^t \left(\int_0^r (v-s)^{(-(\alpha_F-\alpha)\rho-1)\frac{q}{q-1}\wedge \frac{q}{1-q}}\,ds\right)^{\frac{q-1}{q}}\left(\int_0^r(1+\EE\|u(s)\|^q)\,ds\right)^{\frac{1}{q}} \,dv\\
&\leq C \int_r^t v^{\left(-(\alpha_F-\alpha)\rho-1+\frac{q-1}{q}\right)\wedge \left(-1+\frac{q-1}{q}\right)}\,dv\leq C(t-r)^{\left(-(\alpha_F-\alpha)\rho+\frac{q-1}{q}\right) \wedge \frac{q-1}{q}},
\end{align*}
provided that $-(\alpha_F-\alpha)\rho+\frac{q-1}{q}>0$, that is, $(\alpha_F-\alpha)\rho<1-\frac{1}{q}.$ Therefore, choosing $\delta$ large enough so that $\delta (-(\alpha_F-\alpha)\rho+\frac{q-1}{q})>1 $, it follows from Kolomogorov continuity theorem, (see, e.g., \cite[Theorem 1.4.1]{Kunita}), that the process $[0,\infty)\ni t\mapsto  \mF u (t)
$ has a continuous modification in $H_{-\alpha}^A$.
\del{First, by Theorem 1, \cite[page 181]{dinculeanu} the process
$$
[0,T]\ni t \mapsto \xi(t) =\int_0^t \CT(-s) G(u_R(s);z)\tilde \eta(dz,ds)
$$
has a \cadlag modification, denoted by $\bar \xi$.
}

It remains to show that
the process
$$[0,\infty)\ni t\mapsto \mG u (t)= \int_0^ {t} \int_Z\CT(t-s) G(s,u (s),z)  \tilde \eta(dz,ds)
$$
has a  \cadlag modification in $H_{-\alpha}^A$, but this follows from Proposition \ref{prop:cadlag} as $\alpha\ge \alpha_G$.
\medskip

\paragraph{\bf Step III}
Finally, we have to show that under the additional condition \eqref{hypogeneral} the random variable $u(t)$ is $H_{-\alpha_I} ^A$--valued for any $t\in[0,T]$.
In fact, we will show that there exists a constant $C(\lambda)>0$ such that for any process $\xi\in\mX_\lambda$ and $t\in(0,T)$
\DEQSZ\label{assertion1}
 \EE |(\mZ \xi) (t)|_{H_{-\alpha_I} ^A }^q &\le & C\lk( 1+ |u_0|_{H_{-\alpha_I}^A}+ C(\lambda) |\xi|_{\mX_\lambda} \rk).
\EEQSZ
%In addition $C(\lambda)\to 0$ as $\lambda\to \infty$.
We have that
\DEQS
\EE |u  (t)|_{H_{-\alpha_I} ^A} ^q&\le & C\lk( \EE |\CT(t) u _0|_{H_{-\alpha_I} ^A}^q + \EE \lk|\int_0 ^t \CT(t-s) F(s,u  (s))\, ds\rk|_{H_{-\alpha_I} ^A}^q\rk.
\\
&&{}\lk. + \EE \lk| \int_0 ^t \int_Z \CT(t-s) G(s,u  (s),z)\tilde \eta (dz,ds)\rk|_{H_{-\alpha_I} ^A}^q\rk) .
\EEQS
Since $\CT(t):H_{-{\alpha_I} }^A \to H_{-\alpha_I} ^A$ is a bounded operator,
$$ \EE |\CT(t) u_0|^q _{H_{-\alpha_I} ^A} \le C |u_0|^q _{H_{-\alpha_I} ^A} .
$$
Next, we will treat the second summand. Here, Minkowski's inequality, Lemma \ref{smoothing} and the Lipschitz property of $F$ taking into account that \eqref{addass} holds
together with H\"older's inequality give
\DEQS
\lqq{ \EE \lk|\int_0 ^t \CT(t-s) F(s,u  (s))\, ds\rk|_{H_{-\alpha_I} ^A}^q}
&&
\\
&\le &  \EE\lk( \int_0 ^t \lk| \CT(t-s) F(s,u  (s))\rk|_{H_{-\alpha_I} ^A}\, ds\rk)^q\\
& \le&  \EE \lk( \int_0 ^t (t-s)^ {-((\alpha_F-\alpha_I)\wedge 0) \rho} \lk|  F(s,u  (s))\rk|_{H_{-\alpha_F} ^A}\, ds \rk)^q
\\
&\le&
C \lk( 1+\EE \int_0 ^t   \lk|  u  (s)\rk|^ q_{H}\, ds
\rk). \EEQS
Since the RHS can be estimated by $C(1+|u  |_{\mathfrak{X}_\lambda}^ q)$ we now bound the third summand.
Inequality \eqref{DD},
gives %By Lemma \ref{smoothing} we get
\DEQS
\lqq{  \EE \lk| \int_0 ^t \int_Z \CT(t-s) G(s,u  (s),z)\tilde \eta (dz,ds)\rk|_{H_{-\alpha_I} ^A } ^q} &&
\\
 &\le&  \EE \int_0 ^t \int_Z |\CT(t-s) G(s,u  (s),z)| _{H_{-\alpha_I} ^A} ^q \, \nu(dz)\, ds
 \\
 &&+{}
 \EE\lk(  \int_0 ^t \int_Z |\CT(t-s) G(s,u  (s),z)| _{H_{-\alpha_I} ^A} ^p \, \nu(dz)\, ds\rk) ^\frac qp
\\
 &\le&  \EE \int_0 ^t \int_Z |\CT(t-s) G(s,u  (s),z)| _{H_{-\alpha_I} ^A} ^q \, \nu(dz)\, ds
  \\
 &&+{}
C(t) \EE \int_0 ^t \lk( \int_Z |\CT(t-s) G(s,u  (s),z)| _{H_{-\alpha_I} ^A} ^p \, \nu(dz)\rk)^\frac qp \, ds.
 \EEQS
Since $\CT(t):H^A_{-\alpha_I}\to H^A_{-\alpha_I}$ is bounded, we get %By Lemma \ref{smoothing} we get
\DEQS
 &&\EE \lk| \int_0 ^t \int_Z \CT(t-s) G(s,u  (s),z)\tilde \eta (dz,ds)\rk|_{H_{-\alpha_I} ^A } ^q\\
& &\quad\le   \EE \int_0 ^t | G(s,u  (s),z)| _{L^q(Z,\nu;H_{-\alpha_I} ^A)}^q \, ds
+\EE  \int_0 ^t  | G(s,u  (s),z)| _{L^p(Z,\nu;H_{-\alpha_I} ^A)} ^q \, ds
.
\EEQS
Finally, using the fact that $\alpha_I\ge \alpha_G $ we get, by the Lipschitz continuity of $G$, that
\DEQS
 &&\EE \lk| \int_0 ^t \int_Z \CT(t-s) G(s,u  (s),z)\tilde \eta (dz,ds)\rk|_{H_{-\alpha_I} ^A } ^q\\
  &&\quad\le  C\,\lk( 1+  \EE \int_0 ^t |u  (s)| _{H} ^q \,  ds\rk) =C\lk( 1+  C(\lambda)|u  |_{\mathfrak{X}_\lambda}\rk),
\EEQS
which completes the proof of assertion \eqref{assertion1} and also the proof of the theorem.

%%%%%%%%%%%%%%%%%%%%%%%%%%%%%%%%%%%%%%%%%%%%%%%%%%%%%%%%%%%%%%%%%%

}

\end{proof}

\section{The stochastic Volterra equations with non-integrable iumps}\label{sec:ubddj}

As in the previous section, let $Z$ and $Z_L$ be two Banach spaces,  $\tilde \eta$ be a compensated Poisson random measure on $Z$ with intensity measure $\nu$, and $\eta_L$ be a Poisson random measure on $Z_L$ independent to $\tilde{\eta}$ with finite intensity measure $\nu_L$. In applications, the first Poisson random measure will be of infinite activity capturing the small jumps, the second one is of finite activity capturing the large jumps. \\

In order to realize the independent random measures we consider $\tilde \eta$ be a compensated Poisson random measure on $(Z\times \RR_+,\mathcal{B}(Z)\otimes \mathcal{B}({\mathbb{R}_+}))$ over $\mathfrak{A} ^S=(\Omega^S,\CF^S,\{\CF^S_t\}_{t\in [0,T]},\PP^S)$ with intensity measure $\nu$ where
$$\CF^S=\sigma \{ \eta(B ,[0,s]): B \in \CB(Z), s\in[0,T] \}$$ and
$$\CF^S_t=\sigma \{ \eta(B ,[0,s]): B \in \CB( Z), s\in[0,t] \}, 0\leq t\leq T.$$
Furthermore, let $\eta_L$ be a Poisson random measure on $(Z_L\times \RR_+,\mathcal{B}(Z_L)\otimes \mathcal{B}({\mathbb{R}_+}))$ over $\mathfrak{A} ^L=(\Omega^L,\CF^L,\{\CF^L_t\}_{t\in [0,T]},\PP^L)$ with finite intensity measure $\nu_L$ where
$$\CF^L=\sigma \{ \eta(B ,[0,s]): B \in \CB(Z_L), s\in[0,T] \}$$ and
$$\CF^L_t=\sigma \{ \eta(B ,[0,s]): B \in \CB( Z_L), s\in[0,t] \},\quad 0\leq t\leq T.$$
Let $\Omega:=(\Omega^S\times\Omega^L)$, $\mathcal{F}:=\CF^S\otimes \CF^L$, $\CF_t:=\CF^S_t\otimes \CF^L_t$ and $P=P^S\otimes P^L$. With an abuse of notation we denote
by $\tilde \eta$ the compensated Poisson random measure on $(Z\times \RR_+,\mathcal{B}(Z)\otimes \mathcal{B}({\mathbb{R}_+}))$ over $(\Omega,\CF,\{\CF_t\}_{t\in [0,T]},\PP)$ defined by $\Omega \ni (\omega^S,\omega^L)\mapsto \tilde \eta(\omega^S,\cdot)$ and by $\eta_L$ the Poisson random measure on $(Z_L\times \RR_+,\mathcal{B}(Z_L)\otimes \mathcal{B}({\mathbb{R}_+}))$ over $(\Omega,\CF,\{\CF_t\}_{t\in [0,T]},\PP)$ with finite intensity measure $\nu_L$ defined by  $\Omega \ni (\omega^S,\omega^L)\mapsto  \eta_L(\omega^L,\cdot)$.

\medskip
In this section, we will show that there exists a unique global mild solution over $(\Omega,\CF,\{\CF_t\}_{t\in [0,T]},\PP)$ even in the case of unbounded jumps; i.e., we consider the equation
\begin{equation}\label{SDE-bigjumps}
\left\{ \begin{aligned} du(t) & = \lk( A\int_0 ^t b(t-s) u(s)\,ds\rk) \, dt  + F(t,u(t))\,dt \\
& {} + \int_Z G(t,u(t)z)\,\tilde \eta(dz,dt)+
\int_{Z_L} G_L(t,u(t),z)\, \eta_L(dz,dt);\,  t\in (0,T],\\
u(0)&=u_0. \end{aligned} \right.
\end{equation}
Here we will use the representation of
the L\'evy process process or compound Poisson process associated with $\eta_L$ in terms of a sum over all its jumps.
If there would be no memory term, one would simply glue together the solutions between the large jumps. Since we have a non trivial memory term, the solution
does not generate a Markov process and it is not possible to glue together the solutions. Another possibility would be to adopt the approach taken in \cite[Section 9.7]{PESZAT+ZABZCYK} and truncate the intensity measure $\nu$ and employ a stopping time argument. However, then the conditions on $G_L$ would have to be strengthen significantly which we would like to avoid. Therefore, we take a different approach and we first proof existence and uniqueness of a solution given the large jumps and show thereafter that the argument remains valid when the jumps are the jumps of a \levy process.

To this end, we suppose that the Poisson random measure $\eta_L$ with finite intensity measure $\nu_L$ on $Z_L$ is constructed the following way (see, \cite[Theorem 6.4]{PESZAT+ZABZCYK}). 
Let $\sigma=\nu_L(Z_L)$, and let $\{ \tau_n:n\in\NN\}$ be a family of independent exponential distributed real-valued random variables on $\Omega ^L$ with parameter $\sigma$. Consider
\DEQSZ\label{stoppingtimes}
 T_n=\sum_{j=1}^n \tau_j,\quad n\in\NN,
\EEQSZ
and let $\{ N(t):t\ge 0\}$ be the counting process defined by
$$ N(t) :=\sum_{j=1}^\infty 1_{[T_j,\infty)}(t),\quad t\ge 0.
$$
Observe that for any $t>0$, $N(t)$ is a Poisson distributed random variable with parameter $\sigma t$.
Let $\{ Y_n:n\in \NN\}$ be a family of independent, $\frac 1 \sigma \nu_L$-distributed, $Z_L$-valued random variables on $\Omega ^L$. 
%the \levy process given by
%$$
%L_L(t)=\int_0^t \int_Z z\,\eta_L(dz,ds),\quad t\ge 0,
%$$
Then,
$$
 \eta_L=\sum_{j=1}^{\infty}\delta_{(Y_j,T_j)}
 $$
is a Poisson random measure with intensity measure $\nu_L$ and
\begin{multline*}
\int_{0}^{t}\int_{Z_L} S(t-s)G_L(s,u(s),z)\, \eta_L (dz,ds)\\=\sum_{i=1}^ {N(t)} 1_{[T_i,T]} (t)\CT(t-T_i)G_L(T^ -_i,u(T ^-_i),Y_i),\quad t\in [0,T].
\end{multline*}

%\begin{definition}
%A function  $L:[0,\infty) \to [0,\infty)$ is called  {\em slow varying
%at infinity}, if for all  $ a
%> 0$,
%$$
%\lim_{x\to\infty }{L(x)\over L(ax)}=1.
%$$
%If
%$$
%\lim_{x\to\infty }{L(x)\over L(ax)}
%$$
%is finite, and does not equal to one, then the function is called {regularly varying function}.
%
%A function  $L:[0,\infty) \to [0,\infty)$ is called  slowly varying (at infinity) with index $\beta$, if
%for all $a> 0$,
%$$
%\lim_{x\to\infty }{L(x)\over L(ax)}=a^\beta.
%$$
%%\misi{Any restrictions on $\beta$?}
%\end{definition}
In the following, we will use the same notation as in the Section \ref{sec:bddmom}.
We introduce the following assumption on the finite intensity measure $\nu_L$.
\begin{ass}\label{slowvarbigjumps}
 For some $C,\beta>0$, the measure  $\nu_L$ satisfies
$$
 \nu_L \lk( \lk\{ z\in Z_L: |z|> x\rk \}\rk)\leq C x^{-\beta}
$$
\end{ass}
%\begin{remark}
%The assumption $\nu_L(Z_L)T<1$ is a technical assumption. One may put easily parts
%of the large jumps to $Z$ do decrease the total mass $\nu_L(Z_L)$, such that this assumptions is satisfied.
%To me more precise, since $\nu_L(Z_L)<\infty$, the measure is tight. Therefore, for any $T>0$ there exists s compact set $B_T\subset Z_L$ such that
%$\nu(Z_L\setminus B_T)< \tfrac 1T$. Let us define
%$$
%Z_L^{small}:= B_T, \quad \mbox{and} \quad Z_L^{large}:= Z_L\setminus B_T,
%$$
%and $\eta^{small}_L:= \eta_L|Big|_{Z_L^{small}}$ and  $\eta^{large}_L:= \eta_L|Big|_{Z_L^{large}}$.
%Now, by the independently scattered property it follows that $\eta_L^{small}$ and $\eta_L^{large}$ are independent.
%In addition $\eta_L^{small}$ has bounded jumps, therefore can be treated by Theorem  \ref{local_ex}.
%
%\end{remark}
We make the following assumption on the mapping $G_L$.
\begin{ass}\label{assbigjumps}
The mapping
\DEQS
G_L: [0,T]\times H_{-\alpha_I} ^A\times Z_L &\longrightarrow &H ^A_{-\alpha_I},
\\
(t,x,z) &\mapsto & G_L(t,x,z),
\EEQS
is continuous and Lipschitz continuous in the second variable
with Lipschitz constant $L_{G_L}(z)$, $z\in Z_L$, uniformly in $t\in [0,T]$, such that $L_{G_L}(z)\leq M(1+\|z\|_{Z_L})$, $z\in Z_L$.
\end{ass}
The main result of this section is as follows.

\begin{theorem}\label{global}
Suppose that Assumption \ref{ass1} and Assumption \ref{ass2} are satisfied and that
 and suppose that $\nu_L$ satisfies Assumption
\ref{slowvarbigjumps}.
In addition, suppose that the data in  \eqref{SDE-bigjumps} satisfy Assumption \ref{hypogeneral}  and Assumption \ref{assbigjumps}.
and that
\DEQSZ\label{addasszweites}
(\alpha_F-\alpha_I)\rho<1-\frac 1q \quad \mbox{{and}}\quad \alpha_I\ge \alpha_G.
\EEQSZ
If
\begin{equation*}
\alpha\ge \alpha_I,\quad \del{\left((\alpha_G-\alpha)+1\right)  \frac {pq}{q-p}< 1}\alpha\ge \alpha_G, \mbox{ and } (\alpha_F-\alpha)\rho<1-\frac{1}{q},
\end{equation*}
then there exists a  mild solution of \eqref{SDE-bigjumps}  in $H_{-\alpha}$. In particular, % such that
$$\PP\lk( u\in \DD([0,T];H_{-\alpha}^ A)\rk)=1.$$
%\item[(b)] If $b$ also satisfies Assumption \ref{ass3}, and  $\alpha\in\RR$ satisfies $\alpha- \alpha_G\le 1$ and
%\begin{equation*}
%\alpha\ge \alpha_I,\quad \left((\alpha_G-\alpha)\rho+1\right)  \frac {pq}{q-p}< 1, \mbox{ and } (\alpha_F-\alpha)\rho<1-\frac{1}{q},
%\end{equation*}
%then there exists a  mild solution of \eqref{SDE-bigjumps}  in $H_{-\alpha}$. In particular, % such that
%$$\PP\lk( u\in \DD([0,T];H_{-\alpha}^ A)\rk)=1.$$
Finally, for any $t\in[0,T]$, $u(t)$ is  a $H_{-\alpha_I} ^A$--valued
random variable.

\del{In addition,
\begin{enumerate}
  \item for $0<m<\beta$ we have $\EE\int_0^ t |u(s)|^m \, ds <\infty$;
  \item for $0<m<\beta$ and $t\in [0,T]$ we have $u(t)\in L^ m(\Omega;H_2^ {-2\alpha}(\CO))$.
\end{enumerate}
}%If, in addition, $G_L: [0,T]\times H_{-\alpha_I} ^A \to L ^q(Z,\nu;H ^A_{-\alpha_I})$ is Lipschitz continuous in the second variable uniformly on $[0,T]$,
%then $u\in \CM ^{q,q} _0( 0,T;H)$.
%
\end{theorem}
\begin{proof}
{\rm %

As indicated above, the proof proceeds in 2 steps.\medskip \\
\noindent {\bf Step I.}
In the first step, we will show that for any deterministic $N<\infty$, $\{ T_1,\ldots,T_N\}\subset [0,T]$, $T_i<T_{i+1}$, $i=1,\ldots,N-1$, and $\{Y_1,\ldots , Y_N\}\subset Z_L$ there exists a process $u$ over $\mathfrak{A}^S$ such that $u$ solves $\PP^S$--a.s. the integral equation
\begin{align}\label{largejumps}
\\
\nonumber
\lqq{
u(t)=S(t)u_0+\int_{0}^{t}S(t-s)F(s,u(s))\,ds}
&&
\\
\nonumber
&&{} +\int_{0}^{t}\int_Z S(t-s)G(s,u(s),z)\,\tilde \eta (dz,ds)
\\
\nonumber
&&{} +
\sum_{i=1}^ N 1_{[T_i,T]} (t)\CT(t-T_i)G_L(T^ -_i,u(T ^-_i),Y_i). %, \quad \textrm{a.s.}
\end{align}
%
%\end{document}

\medskip
In this step we will use the notation $\mathbb{E}$ for the expectation over $\Omega^S$ to shorten notation.
Using the notation introduced in Step I of the proof of Theorem \ref{local_ex} we
define the integral operator
$\mathfrak{Z}_0$  by
\DEQSZ\label{mZlarge}
\\
\nonumber
\lqq{ (\mathfrak{Z}_0(\xi;(x_i)_{i=1}^ N)(t):= (\mT u_0)(t)+
\lk(\mathfrak{F}\xi\rk) (t)}
\\ &&{}\nonumber+
\lk( \mathfrak{G}\xi\rk)(t) + \sum_{i=1}^ N 1_{[T_i,T]}(t) (\mT_{T_i}  G_L(T^-_i,x_i,Y_i))(t)
, \quad t\in[0,T],\, \xi\in\mX_\lambda.
  \EEQSZ
Here
$\mT_{T_i}:H_{-\alpha_I}^ A \to \mathfrak{X}_\lambda$
is defined by
$$(\mT_{T_i}  x)(t) := \CT(t-T_i) 1_{[T_i,T]}(t) x,
\quad  x\in  H_{-\alpha_I}^ A.
$$

Next, we define a space and an operator to apply a fixed point argument. In order to do this,
 let
\DEQS
\mathfrak{L}_\lambda(T_1,\ldots, T_N):=\lk\{
(x_i)_{i=1}^ N : x_i\in  L^ q(\mathfrak{A}^ S;H_{-\alpha_I}^A),\rk.\hspace{4cm}
\\
\hspace{1cm}\lk. \,\mbox{ for all $i=1,\ldots, N$ the random variable  } x_i\,\mbox{ is $\CF_{T_i^ -}^S$ measurable }\rk\},
\EEQS
with norm
$$\lk|(x_i)_{i=1} ^N \rk|_{\mL_\lambda}:= \lk( \sum_{i=1} ^N e ^{-\lambda T_i}  \, \EE |x_i|_{H_{-\alpha_I} ^A} ^q
\rk) ^\frac 1q .
$$
Secondly, let
$$
X(\lambda,K):= \mathfrak{X}_\lambda\times \mathfrak{L}_\lambda(T_1,\ldots, T_n),
$$
with norm
$$
\lk|(\xi, (x_i)_{i=1}^ N) \rk|_{X(\lambda,K)} := K\, |\xi|_{\mX_\lambda} + \lk|(x_i)_{i=1} ^N \rk|_{\mL_\lambda},\quad  (\xi, (x_i)_{i=1}^ N )\in X(\lambda,K).
$$

\noindent
Let us define a operator $\Theta$ acting on $X(\lambda,K)$
 by putting
$$\Theta(\xi;(x_i)_{i=1}^ N) = (\mathfrak{Z}_0(\xi,(x_i)_{i=1}^ N) %(\mathfrak{Z}_0(\xi)(T_j))_{j=1}^N)
; (\mathfrak{Z}_0(\xi,  (x_i)_{i=1}^ N) (T_j ^-))_{j=1}^N),
$$
for $(\xi,(x_i)_{i=1}^ N) \in X(\lambda,K)$. %\mathfrak{X}_\lambda\times \mathfrak{L}_\lambda$.

\medskip

We will show in the  first part, that for any $K>0$ and $\lambda>0$ the mapping $\Theta$ maps $X(\lambda,K)$ into itself, and, in the second part,  that there exist numbers
$K,\lambda>0$ such that
$\Theta$ is a contraction on $X(\lambda,K)$.

\medskip
First note that there is a modification of $t\to \mathfrak{Z}_0(\xi,  (x_i)_{i=1}^ N) (t)$ that is \cadlag in $H_{-\alpha}$ (in particular, we may take $\alpha=\alpha_I$)  when $\xi \in \mathfrak{X}_\lambda$ and $(x_i)_{i=1}^N\in \mathfrak{L}_\lambda$ as Step II of Theorem \ref{local_ex} together with the assumption on $G_L$ show. Therefore, the one sided limits $\mathfrak{Z}_0(\xi,  (x_i)_{i=1}^ N) (T_j ^-)$ exist in $H_{-\alpha_I}$.

Similarly as before one can show that
first $\mathfrak{Z}_0(\cdot,(x_i)_{i=1}^N)$ maps $\mathfrak{X}_\lambda$ into $\mathfrak{X}_\lambda$ for given $(x_i)_{i=1}^N\in \mathfrak{L}_\lambda$.
First, note, by estimate \eqref{operatorCT} we have $$|\mT u_0|_{\mathfrak{X}_\lambda}\le C\, \lambda ^ {\frac 1q -\alpha_I\rho}|u_0|_{H_{-\alpha_I}^A}.$$
Next,
Corollary \ref{det_st}, Proposition \ref{stconv} and the  Lipschitz continuity of $F$ and $G$ imply that the operators
$ \mathfrak{F}$ and $\mathfrak{G}$  map $\mX_\lambda$ into itself.
Finally,  from estimate \eqref{operatorCT} and Assumption \ref{assbigjumps},
it follows that for each $i=1,\ldots,N$  the terms  $\mT_{T_i}G_L(T_i^-,x_i,Y_i) $ belong to $\mX_\lambda$.
Indeed, one can show by similar calculation as done for estimate \eqref{operatorCT}, that for $(y_i)_{i=1}^ N$ the following estimate holds:
\DEQS
\lqq{ \lk\| \mT _{T_i} y_i\rk\|^ q_{\mathfrak{X}_\lambda } \le \int_{T_i}^T e^ {-\lambda t}|\CT(t-T_i) y_i|_H^ q \, dt }
&&
\\&\le & C e^ {-\lambda T_i}
 \int_{T_i}^T e^ {-\lambda (t-T_i)} (t-T_i) ^ {-\rho\alpha_Iq} \, dt\, | y_i|_{H_{-\alpha_I}^ A}^q.
\EEQS
Therefore,
\DEQS
\EE  \lk\| \mT _{T_i} y_i \rk\|_{\mathfrak{X}_\lambda} ^ q &\le &
 C\Gamma( 1-\rho\alpha_Iq)\, \lambda ^ {\rho  \alpha_I q-1}\, e^ {-\lambda T_i} \EE |y_i|^ q _{H_{-\alpha_I}^ A}
. %\le C(\lambda) \, e ^{-\lambda T_i}\EE |x_i|^ q _{H_{-\alpha_I}^ A}.
\EEQS
Taking the sum over $i=1,\ldots,N$
we get
\DEQS
\sum_{i=1}^ N\EE  \lk\| \mT _{T_i} y_i \rk\|_{\mathfrak{X}_\lambda} ^ q &\le &
 C\Gamma( 1-\rho\alpha_Iq)\, \lambda ^ {\rho  \alpha_I q-1}\,\sum_{i=1}^ N e ^ {-\lambda T_i} \EE |y_i|^ q _{H_{-\alpha_I}^ A} .
\EEQS
Putting $y_i=G_L(T_i^-,x_i,Y_i)$ and using Assumption  \ref{assbigjumps} we get
\DEQS
&&\left\|\sum_{i=1}^ N 1_{[T_i,T]}(\cdot) (\mT_{T_i}  G_L(T^-_i,x_i,Y_i))(\cdot)\right\|_{\mathfrak{X}_\lambda}^q\le C\sum_{i=1}^ N\EE  \lk\| \mT _{T_i} G_L(T_i^-,x_i,Y_i) \rk\|_{\mathfrak{X}_\lambda} ^ q
\\
& &\le
 C\Gamma( 1-\rho\alpha_Iq)\, \lambda ^ {\rho  \alpha_I q-1}\,\sum_{i=1}^ N e ^ {-\lambda T_i}(1+ L_{G_L}(Y_i))^q\EE |x_i|^ q _{H_{-\alpha_I}^ A},
\EEQS
%
%
%Replacing $(x_i)_{i=1}^N$ by $  (\mathfrak{Z}_0(\xi,  (x_i)_{i=1}^ N) (T_j ^-))_{j=1}^N)$ and taking into account estimate
where $C$ depends on $N$. Thus,
we have shown that
$$\lk| \mathfrak{Z}_0 (\xi,(x_i)_{i=1}^ N) \rk|_{\mX_\lambda}\le
C\lk( |u_0|_{H_{-\alpha_I}^ A} + |\xi|_{\mX_\lambda}+ \lk|(\xi,(x_i)_{i=1}^ N\rk| _{\mL_\lambda}  )\rk).
$$

\medskip

It remains to show that
$$
\sum_{j=1}^N
e ^{-\lambda T_j} \EE |  (\mathfrak{Z}_0(\xi, (x_i)_{i=1}^ N)(T_j ^-)|^ q _{H_{-\alpha_I}^A}<\infty.
$$
But this follows by estimate \eqref{assertion1}  in Step III of the  proof of Theorem \ref{local_ex}.

\medskip

Next, we show that there exist numbers  $\lambda>0$ and $K>0$  such that $\Theta$ is a contraction on $H(\lambda,K)$.
First, we will analyse $\mZ_0(\xi,(x_i)_{i=1}^N)$.
Let $\xi_1,\xi_2\in\mX_\lambda$ and $(x_i^1)_{i=1}^N, (x_i^2)_{i=1}^N\in \CL_\lambda$.
Taking the difference $ \mathfrak{Z}_0(\xi _1,(x_i^1)_{i=1}^N) -\mathfrak{Z}_0(\xi_2,(x_i^2)_{i=1}^N)$ we see that $\mT u_0$ will disappear.
Similarly as in Step I of the proof of Theorem \ref{local_ex} we know from %
Corollary  \ref{det_st} and the  Lipschitz continuity of $F$  that the operator
\DEQSZ\label{operatorflarge}
 \mathfrak{F} :\mX_\lambda \longrightarrow \mX_\lambda.
\EEQSZ
is Lipschitz continuous  with Lipschitz constant $L_F\, C(\alpha_F,\lambda)$ with
$$\lim_{\lambda\to \infty }  C(\alpha_F,\lambda)=0.
$$
Again, Proposition \ref{stconv} and the  Lipschitz continuity of $G$ implies that the operator
\DEQSZ\label{operatorglarge}
 \mathfrak{G} :\mX_\lambda \longrightarrow \mX_\lambda.
\EEQSZ
is Lipschitz continuous as well with Lipschitz constant $L_G\, C(\alpha_G,\lambda)$ with
$\lim_{\lambda\to 0}  C(\alpha_G,\lambda)=0$.
It remains to consider the mappings
$$
 t\mapsto 1_{[T_i,T]}(t) \mT_{T_i}  G_L(T_i^{-},x_i,Y_i)(t)
, \quad t\in[0,T],\, \xi \in\mX_\lambda, \, i=1,  \ldots ,N.
$$
Similar calculation as for \eqref{operatorCT} gives for $(y_i)_{i=1}^ N$
\DEQS
\lqq{ \lk\| \mT _{T_i} y_i\rk\|^ q_{\mX_\lambda  } \le \int_{T_i}^T e^ {-\lambda t}|\CT(t-T_i) y_i|_H^ q \, dt }
&&
\\&\le & e^ {-\lambda T_i}
 \int_{T_i}^T e^ {-\lambda (t-T_i)} (t-T_i) ^ {-\rho\alpha_Iq} \, dt\, | y_i|_{H_{-\alpha_I}^ A}^q
\EEQS
and therefore,
\DEQSZ\label{operatorCTlarge}
\lk\| \mT _{T_i} y_i\rk\|_{\mX_\lambda  }^q &\le &
 \Gamma( 1-\rho\alpha_Iq) \, \lambda ^ {\rho  \alpha_I q-1}\, e ^{-\lambda T_i} |y_i |^q_{H_{-\alpha_I}^ A}.
\EEQSZ
Thus, setting $y_i= G_L(T_i^-,x_i,Y_i)$, we get
\DEQS\lqq{
 \|  1_{[T_i,T]}(\cdot) \mT_{T_i}  G_L(T_i^-,x^ 1_i,Y_i)(\cdot)- 1_{[T_i,T]}(\cdot) \mT_{T_i}  G_L(T_i^-,x_i^ 2,Y_i)(\cdot)\|^q_{\mathfrak{X}_\lambda}}
\\
& \le&   C\, \lambda ^ {\rho\alpha_Iq-1 } \, e ^{-\lambda T_i}\EE |G_L(T_i^-,x^ 1_i,Y_i)-G_L(T_i^-,x_i^ 2,Y_i)|^q_{H_{-\alpha_I}^A} .
\EEQS
The Lipschitz property of $G_L$ in the second variable gives
\DEQS\lqq{
\|  1_{[T_i,T]}(\cdot) \mT_{T_i}  G_L(T_i^-,x^ 1_i,Y_i)(\cdot)- 1_{[T_i,T]}(\cdot) \mT_{T_i}  G_L(T_i^-,x_i^2,Y_i)(\cdot)\|^q_{\mathfrak{X}_\lambda}}
\\
& \le&     C\, \lambda ^ {\rho\alpha_Iq-1 } \,e ^{-\lambda T_i} L_{G_L}(Y_i)^ q\EE |x_i^1-x_i^2|_{H_{-\alpha_I}^A}^ q .
\EEQS
Taking the sum over $i=1,\ldots,N$ gives
\DEQS
&&\left\|\sum_{i=1} ^N   1_{[T_i,T]}(\cdot) \mT_{T_i}  G_L(T_i^-,x_i^ 1,Y_i)- \sum_{i=1} ^N1_{[T_i,T]}(\cdot) \mT_{T_i}  G_L(T_i^-,x_i^ 2,Y_i)\right\|^q_{\mathfrak{X}_\lambda}\\
&&\le C
\sum_{i=1} ^N  \|  1_{[T_i,T]}(\cdot) \mT_{T_i}  G_L(T_i^-,x_i^ 1,Y_i)- 1_{[T_i,T]}(\cdot) \mT_{T_i}  G_L(T_i^-,x_i^ 2,Y_i)\|^q_{\mathfrak{X}_\lambda}
\\
&&\le  \max(L_{G_L}(Y_1)^q,\ldots, L_{G_L}(Y_N)^q)    C\, \lambda ^ {\rho\alpha_I q-1 } \, |(x ^1_i)_{i=1} ^N -(x ^2_i)_{i=1} ^N |_{\mL_\lambda }^ q ,
\EEQS
where $C$ depends on $N$. Therefore, we have shown that
%\end{document}
\DEQS
\lqq{ \lk| \mathfrak{Z}_0 (\xi _1,(x_i^1)_{i=1}^N) -\mathfrak{Z}_0 (\xi_2,(x_i^2)_{i=1}^N)\rk|_{\mX_\lambda} }&&
\\
&\le &
C_0(\lambda)\,\lk(  |\xi _1-\xi_2|_{\mX_\lambda} + |(x_i^1)_{i=1}^N-(x_i^2)_{i=1}^N)|_{\mathfrak{L}_\lambda}\rk) ,
\EEQS
with $C_0(\lambda)\to0$ as $\lambda\to\infty$.
%\end{document}
\medskip
It remains to consider
\DEQS
 |(\mathfrak{Z}_0(\xi_1, (x^ 1_i)_{i=1}^ N)(T ^-_j) )_{j=1}^N) -(\mathfrak{Z}_0(\xi_2, (x^ 2 _i)_{i=1}^ N)(T^-_j))_{j=1}^N) |_{\mathfrak{L}_\lambda}
 .\EEQS

Again, we begin with mimicking the calculations done in Step III of the proof of Theorem \ref{local_ex}. In particular,
\DEQS
\lqq{ e ^ {-\lambda T_j } \EE |\mZ_0(\xi ^1, (x ^1 _i)_{i=1} ^N) ({T_j ^-})-\mZ_0(\xi ^2 , (x ^2 _i)_{i=1} ^N) ({T_j ^-})|_{H_{-\alpha_I} ^A} ^q}
&&
\\ &\le & C\lk(  e ^ {-\lambda T_j }\EE \lk|\int_0 ^{T_j ^-} \CT({T_j ^-}-s)\lk(  F(s,\xi  ^1(s))-F(s,\xi ^2  (s) )\rk)\, ds\rk|_{H_{-\alpha_I} ^A}^q\rk.
\\
&&{}\lk. +  e ^ {-\lambda T_j }\EE \Big| \int_0 ^{T_j ^-} \int_Z \CT({T_j ^-}-s)\rk.
\\
&&\lk.\lk(  G(s,\xi ^1  (s),z)- G(s,\xi  ^2 (s),z)\rk) \tilde \eta (dz,ds)\Big|_{H_{-\alpha_I} ^A}^q\rk.
\\
&&{}\lk. + e ^ {-\lambda T_j }\sum_{i=1} ^N 1_{[T_i,T]}({T_j ^-})\,  \rk.
\\
&&\lk.\EE \lk|\CT({T_j ^-}-T_i)\lk( G_L(T_i,x ^1_i,Y_i)- G_L(T_i,x ^2_i,Y_i)\rk) \rk|_{H_{-\alpha_I} ^A}^q\rk).
\EEQS
We are going to estimete the two summands separately. First, we get
 for a fixed $T_j$ that
\DEQS
\lqq{ e ^ {-\lambda T_j } \EE \Big|\int_0 ^{T_j ^-} \CT({T_j ^-}-s)\lk( F(s,\xi ^1 (s))-F(s,\xi ^2 (s))\rk) \, ds\Big|_{H_{-\alpha_I} ^A}^q
} &&
\\
& \le & \EE \Bigg( \int_0 ^{T_j ^-}e ^ {-\lambda (T_j-s)/q } ({T_j ^-}-s)^ {-((\alpha_F-\alpha_I)\wedge 0) \rho}e ^ {-\lambda s/q }
\\
&&\lk|  F(s,\xi  ^1 (s))-F(s,\xi ^2
(s))\rk|_{H_{-\alpha_F} ^A}\, ds \Bigg)^q
\\
&\le &
C(\lambda)\,  \EE \int_0 ^{T_j ^-}  e ^ {-\lambda s } \lk|  \xi ^1 (s)- \xi ^2 (s)\rk|^ q_{H}\, ds
\le C(\lambda) |\xi_1-\xi_2|_{\mX_\lambda} ^q,
\EEQS
with $C(\lambda)\to 0$ as $\lambda\to\infty$.
Similarly we get for the second summand using that $\alpha_I\ge \alpha_G$,
\DEQS
\lqq{ e ^ {-\lambda T_j }  \EE \lk| \int_0 ^{T_j ^-} \int_Z \CT({T_j ^-}-s) \lk( G(s,\xi  ^1(s);z)-G(s,\xi  ^2 (s);z)\rk) \tilde \eta (dz,ds)\rk|_{H_{-\alpha_I} ^A } ^q}
 &&
\\
 &\le&  e ^ {-\lambda T_j } \Bigg[
 \\
 &&{}
 \EE \int_0 ^{T_j ^-} \int_Z |\CT({T_j ^-}-s)\lk(  G(s,\xi ^1 (s);z)-G(s,\xi ^2  (s);z)\rk)| _{H_{-\alpha_I} ^A} ^q \, \nu(dz)\, ds\\
&+&  \EE\left( \int_0 ^{T_j ^-} \int_Z |\CT({T_j ^-}-s)\lk(  G(s,\xi ^1 (s);z)-G(s,\xi ^2  (s);z)\rk)| _{H_{-\alpha_I} ^A} ^p \, \nu(dz)\, ds\right)^{\frac{q}{p}}\Bigg]
 \EEQS
 \DEQS&\le & C(1+C(T))
 \\
 &&{}\times
  \EE \int_0 ^{T_j ^-}  e ^ {-\lambda (T_j-s) }  e ^ {-\lambda s } |\xi  ^1(s)-\xi  ^2(s)| _{H} ^q \, ds \le  C |\xi ^1 -\xi ^2  |_{\mathfrak{X}_\lambda}^q .
\EEQS
Finally, using the Lipschitz property of $G_L$, we have
\DEQS
\lqq{  e^{-\lambda T_j}  \EE\Big| \sum_{i=1} ^N   1_{[T_i,T]}(T_j ^-) \mT_{T_i}  G_L(T_i^-,x ^1_i,Y_i)(T_j^-)
}
&&
\\
&&{} - 1_{[T_i,T]}(T_j ^-) \mT_{T_i}  G_L(T_i^-,x ^2_i,Y_i)(T_j^-)\Big|_{H_{-\alpha_I} ^A} ^q
\\
&\le & C   \EE\sum_{i=j+1}^ N e ^{-\lambda (T_j-T_i) }e ^{-\lambda T_i}\lk|  \mT_{T_i} \lk(  G_L(T_i^-,x ^1_i,Y_i)  - G_L(T_i^-,x ^2_i,Y_i)\rk)(T_j^-)\rk|_{H_{-\alpha_I} ^A} ^q\\
 &\le &  C e ^{-\lambda \min_{i=j+1,\ldots,N} (T_j-T_i) }
       \EE\sum_{i=j+1}^ N e ^{-\lambda T_i}L_{G_L}(Y_i)^q\lk|   x ^1_i-x ^2_i \rk|_{H_{-\alpha_I} ^A} ^q,
\EEQS
where $C$ depends on $N$.

Collecting the estimates and summing up over $j=1,\ldots,N$ we get
\DEQS
\lqq{ |(\mathfrak{Z}_0(\xi_1, (x^ 1_i)_{i=1}^ N)(T ^-_j) )_{j=1}^N) -(\mathfrak{Z}_0(\xi_2, (x^ 2 _i)_{i=1}^ N)(T^-_j))_{j=1}^N) |_{\mathfrak{L}_\lambda }^q}
&&\\
&= &  \sum_{j=1}  ^N
\EE e  ^{-\lambda T_j}\lk| \mathfrak{Z}_0(\xi_1, (x^ 1_i)_{i=1}^ N)(T ^-_j) -\mathfrak{Z}_0(\xi_2, (x^ 2 _i)_{i=1}^ N)(T^-_j) \rk|_{H_{-\alpha_I} ^A} ^q
\\
&\le & C(\lambda,N,T )\left( |\xi_1-\xi_2|_{\mX_\lambda} ^q\right.
\\&&{}
% % % % % % % % % % % % %
\left.+ \sum_{j=1}  ^N e ^{-\lambda \min_{i=j+1,\ldots,N} (T_j-T_i) }
       \EE\sum_{i=j+1}^ N e ^{-\lambda T_i}L_{G_L}(Y_i)^q\lk|   x ^1_i-x ^2_i \rk|_{H_{-\alpha_I} ^A} ^q.\right)\\
&&\le C(\lambda,N,T ) \left( |\xi_1-\xi_2|_{\mX_\lambda} ^q\right.\\
&&\left.{} +N e ^{-\lambda \min_{i,j=1,\ldots,N\atop j\not = i } (T_j-T_i) }
\max(L_{G_L}(Y_1)^q,\ldots,L_{G_L}(Y_N)^q)\,
\rk.\\&&\lk.       \EE\sum_{i=1}^ N e ^{-\lambda T_i}\lk|   x ^1_i-x ^2_i \rk|_{H_{-\alpha_I} ^A} ^q\right)
\EEQS
Where $1\le C(\lambda,T,N)\le C$ for all $\lambda>1$.
In summary, we obtain
\DEQS
\lqq{\lk| \Theta(\xi_1,(x_i^ 1 )_{i=1}^ N)- \Theta(\xi_2,(x_i^ 2 )_{i=1}^ N)\rk|_{H(\lambda,K)}
} &&
\\
&\le & \lk[ C_0(\lambda) +\frac {C(\lambda,T,N)^{\frac{1}{q}}} {K} \rk]
 K\, |\xi_1-\xi_2|_{\mX_\lambda}
\\
&&{}+ \left[ C_0(\lambda) K+  C(\lambda,T,N)^{\frac{1}{q}} N^{\frac{1}{q}} e ^{-\frac{\lambda}{q} \min_{i,j=1,\ldots,N\atop j\not = i } (T_j-T_i) }\right.\\
&&\left.{}\max(L_{G_L}(Y_1),\ldots,L_{G_L}(Y_N))\right]\quad\times|(x_i^ 1 )_{i=1}^ N-
(x_i^ 2 )_{i=1}^ N|_{\mL_\lambda}.
\EEQS
If $K$ is chosen such that $\frac {C(\lambda,T,N)^{\frac{1}{q}}} {K}<\frac 12$ for all $\lambda>1$ and then
$\lambda>1$ is chosen large enough so that $C_0(\lambda)\le \frac 12$, and
\begin{multline*}
C_0(\lambda) K+  C(\lambda,T,N)^{\frac{1}{q}} N^{\frac{1}{q}} \\e ^{-\frac{\lambda}{q} \min_{i,j=1,\ldots,N\atop j\not = i } (T_j-T_i) }
\max(L_{G_L}(Y_1),\ldots,L_{G_L}(Y_N)) < 1,
 \end{multline*}
 then there exists a number $0<k<1$ such that
\begin{equation}\label{eq:contr}
|\Theta(\xi_1, (x^ 1_i)_{i=1}^ N) -\Theta (\xi_2, (x^ 2 _i)_{i=1}^ N )|_{H(\lambda,K)}
\le k |(\xi_1-\xi_2,(x_i ^1 -x_i ^2 )_{i=1} ^N)|_{H(K,\lambda)}.
\end{equation}
Thus, we have proven that there exist numbers $\lambda,K>0$ such that  $\Theta$ is a strict contraction on $H(\lambda,K)$.
Applying Banach's fixed point theorem, it follows that there exists a process $\xi^\ast \in \mathfrak{X}_\lambda$ with a \cadlag modification in $H_{-\alpha}$ such that
\begin{align}
\lqq{
\xi^\ast (t)=S(t)u_0+\int_{0}^{t}S(t-s)F(s,\xi^\ast (s))\,ds}
&&
\\
\nonumber
&&{} +\int_{0}^{t}\int_Z S(t-s)G(s,\xi^\ast (s),z)\,\tilde \eta (dz,ds)
\\
\nonumber
&&{} +
\sum_{i=1}^ N 1_{[T_i,T]} (t)\CT(t-T_i)G_L(T^ -_i,\xi^\ast (T ^-_i),Y_i). %, \quad \textrm{a.s.}
\end{align}
%
%It remains to show that there exists a \cadlag modification of $\xi^\ast$. Mimicking the calculation in Step II
%in the proof of Theorem \ref{local_ex}, one can see that on the time interval
%$[0,T_1)$, there exists a \cadlag modification $\bar \xi^ \ast|_{[0,T_1)}$ of $\xi^ \ast$  in $H_{-\alpha}$.
%Since $x_1\in H_{-\alpha_I}$, the mapping $\RR_0^ +\ni t \mapsto S(t)x_1$ is continuous  in $H_{-\alpha}$,
%and again mimicking the calculation in Step II
%in the proof of Theorem \ref{local_ex},  one can see that on the time interval
%$[0,T_2)$, there exists a \cadlag modification $\bar \xi^ \ast|_{[0,T_2)}$ of $\xi^ \ast$  in $H_{-\alpha}$.
%In this way, one can see that the fix point is a solution to equation \eqref{largejumps} in $H_{-\alpha}$.
%
%
%\end{document}

%\paragraph{\bf Step II:}

\noindent {\bf Step II.}
%As we mentioned on page \pageref{psplit} the $\sigma$--algebras $\CF^L$ and $\CF^S$ are independent.
%Therefore, for all $U_L\in\CF^L$ and $U_S\in \CF^S$, $\PP\lk( U_L\cap U_S\rk)=\PP(U_L)\,\PP(U_S)$.
%In addition, $\CF=\sigma(\CF^L \wedge \CF^S)$.
Fix $N,n\in\NN$ and put
\begin{multline*}
A_n ^N :=
 \left\{ \omega^L\in \Omega^L: N(T)=N; |Y_i|_{Z_L}\le n,~ i=1,...,N;\right.\\
 \left.\min_{  1\le i,j\le N\atop i\not=j}|T_i-T_j| \ge \frac Tn\right\} \in\CF^L,
\end{multline*}
where $\{T_j:j=1\le j\le N\}$ and $\{Y_j:1\le j\le N\}$ are defined in \eqref{stoppingtimes} and \eqref{eq:lp}.
As the contraction constant $k$ in \eqref{eq:contr} only depends on $N$ and $n$ and $L_{G_L}(z)\leq M(1+|z|_{Z_L})$, Step I also show that, there is $0<k<1$, such that
$$
\sup_{\omega^L\in A_n ^N}|\Theta(\xi_1, (x^ 1_i)_{i=1}^ N) -\Theta (\xi_2, (x^ 2 _i)_{i=1}^ N )|_{H(\lambda,K)}
\le k |(\xi_1-\xi_2,(x_i ^1 -x_i ^2 )_{i=1} ^N)|_{H(K,\lambda)}.
$$
where in the definition of $\Theta$ we allow $\{T_j:j=1\le j\le N\}$ and $\{Y_j:1\le j\le N\}$  be random variables from \eqref{stoppingtimes} and \eqref{eq:lp}. Therefore, Banach's fixed point theorem on $L^{\infty}(A_n ^N; H(\lambda,K))$ shows that there is a unique progressively measurable process $\xi_n^N$ on $\Omega^S\times A_n^N $  with a \cadlag modification in $H_{-\alpha}$  such that
for $\mathbb{P}$ almost surely on $\Omega^S\times A_n^N $, the process $\xi_n^N$ satisfies \eqref{largejumps}. We extend $\xi_n^N$ by setting it $0$ on  $\Omega^S\times \Omega^L\setminus \Omega^S\times A_n^N $.
Let $\gamma$ be a number such that $0<\gamma< \min(\frac 12 ,\beta)$,
and put   $g(n)=\lfloor n^ \gamma \rfloor$. Let $B_n:= \cup_{N=1 }^ {g(n)}  A_n ^N$ with $B_0:=\emptyset$.  Note that for fixed $n$, the sets $\{ A_n ^N,~N=1,...,g(n)\}$ are disjoint and that $B_n\subset B_{n+1}$.
Define
\begin{equation*}
u(t):=\sum_{n=1}^{\infty}1_{B_n\setminus B_{n-1}}\sum_{N=1}^{g(n)} \xi_n^N(t)
\end{equation*}
Then $u$ satisfies \eqref{largejumps} for $\mathbb{P}$ almost all $\omega \in \Omega^S\times (\cup_{n=1}^{\infty} B_n) $.
%Let $B_n:= \cup_{N=1 }^ {g(n)}  A_n ^N$.  Note that for fixed $n$, the sets $\{ A_n ^N,~N=1,...,g(n)\}$ are disjoint and that $B_n\supset B_{n+1}$.
%It is not hard to see that $u_n(t)=u_m(t)$ for $n<m$ on $B_n$, $\mathbb{P}$ almost surely,
%and therefore the limit in \eqref{eq:u} exists.
It remains to show that
 $\lim_{n\to\infty} \PP^L( B_n)=1$.

First, note that given $N(T)=N$, the times $\{T_i:i=1,\ldots, N\}$ are uniformly distributed on the interval $[0,T]$
(see  \cite[Proposition 2.9]{tankov})
and are independent of  $\{ Y_j:j\ge 0\}$.
Fix $n\in\NN$. To give an lower estimate of $\PP^L(B_n)$, observe that
$$
\PP^L\lk( |T_1-T_2|\ge \frac T n\mid N(T)=N \rk) \ge \lk(1-\frac 2n\rk).
$$
Throwing $T_3$ into the interval $[0,T]$, by the Bayes formula, we get
\DEQS
\lqq{ \PP^L\lk( \min_{i,j=1,2,3\atop i\not = j}|T_i-T_j|\ge \frac T n \mid N(T)=N \rk)}
&& \\
& =&\PP^L\lk( |T_1-T_2|\ge \frac T n\mid N(T)=N \rk)
\\&&{}\times \PP^L\lk( \min_{i=1,2} |T_3-T_i|\ge\frac Tn\,\Big|\, |T_1-T_2|\ge \frac T n \mbox{ and }  N(T)=N \rk)
\\
&\ge & (1-\frac 2n) (1-\frac 4n).
\EEQS
Iterating, we get
\DEQS
 \PP^L\lk( \min_{i,j=1,\ldots,N\atop i\not = j}|T_i-T_j|\ge \frac T n\mid N(T)=N \rk)
\ge \prod_{j=1}^{N} (1-\frac {2j}n).
\EEQS
Since
$$-\log\lk( \prod_{j=1}^{N} (1-\frac {2j}n)\rk)= -\sum_{j=1} ^N \log\lk( 1-\frac {2j}n\rk)\le  \sum_{j=1} ^N \frac {2j}n \simeq {N(N-1)\over n}.
$$
there exists a constant $c>0$ such that for $n\ge \sqrt{N}$
\begin{equation}\label{eq:les}
 \PP^L\lk( \min_{i,j=1,\ldots,N\atop i\not = j}|T_i-T_j|\ge \frac T n\mid N(T)=N \rk)
\ge e ^{-c{N(N-1)\over 2n}}.
\end{equation}
%Therefore, setting $N=g(n)$
%$$\lim_{n\to\infty}  \PP\lk( \min_{i,j=1,\ldots,N\atop i\not = j}|T_i-T_j|\ge \frac 1 n\rk)
% = 1.
% $$

%Next, put $Y_N ^\ast  = \min _{1\le j\le N} |Y_j|_Z$.
%Recall that that  $\nu_L$ is slowly varying with index $\beta$.
%Hence, it follows from the Fisher-Tippet Theorem (compare Appendix \ref{fisher}) that
%  $$ c_N ^{-1} Y^\ast _N \stackrel{d}{\to}  \Phi_\beta,
%  $$
%where
%   $$ c_N = F^{ \leftarrow}(1-N^{-1})= \inf_{x\in\RR} \{ \nu_L \lk(\{ z\in Z_L: |z|\ge x\} \rk)/\sigma  \ge 1-N ^{-1}\}.
%  $$
%In particular, by Assumption \ref{slowvarbigjumps} it follows that $c_N\sim N ^\frac 1\beta$.

%
\medskip
As noted before, the sets $\{ A_n ^N,~N=1,...,g(n)\}$ are disjoint for fixed $n$ and
therefore,
\begin{align*}
&\PP^L\lk( B_n\rk) = \sum_{N=1} ^{g(n)} \\
&  \PP^L\lk(  \lk\{ \omega^L\in \Omega^L : N(T)=N; |Y_i|_{Z_L}\le n,~ i=1,...,N;\min_{  1\le i,j\le N\atop i\not=j} |T_i-T_j|\ge \frac Tn\rk\}\rk)\\
&=\sum_{N=1}  ^{g(n)}  \PP^L\lk(\lk\{ \omega^L\in \Omega^L :  N(T)=N\rk\}\rk)\\
&   \PP^L\lk(  \lk\{ \omega^L\in \Omega^L :  |Y_i|_{Z_L}\le n,~ i=1,...,N;\min_{  1\le i,j\le N\atop i\not=j} |T_i-T_j| \ge \frac Tn\rk\}\Big| N(T)=N\rk)
\end{align*}
Under the condition  $N=N(T)$ the random variables $T_i$, $i=1,\ldots, N$ and $Y_i$, $i=1,\ldots, N$, are mutually independent. In addition, $N(T)$ is a Poisson distributed random variable with parameter $\sigma T=\nu_L(Z_L)\,T$ and $N(T)$ is independent of $Y_i$, $i=1,\ldots, N$.
Therefore, using Assumption  \ref{slowvarbigjumps}, estimate \eqref{eq:les} together with Bernoulli's inequality, we get for $n$ large enough that
\DEQS
\PP^L\lk( B_n\rk)
&\ge &  \sum_{N=1} ^{g(n)}  \PP^L\lk( \lk\{ \omega^L\in \Omega^L :  N(T)=N\rk\}\rk) e ^{-c{N(N-1)\over 2n}}  \lk(1 -C n ^{-\beta } N\rk)
\\
&\ge &  \sum_{N=1} ^{g(n)}    e ^{-\sigma T} {(\sigma T) ^N \over N!
}
  e ^{-c{N(N-1)\over 2n}}  \lk(1 - Cn ^{\gamma-\beta } \rk)
\\
&\ge &  \sum_{N=1} ^{g(n)}    e ^{-\sigma T} {(\sigma T) ^N \over N!
}
  e ^{-c{g(n)^ 2 \over 2n}}  \lk(1 - Cn ^{\gamma-\beta } \rk)
\\
&=&e ^{-\frac{c}{2}{n ^{2\gamma-1} }} \lk(1 - Cn ^{\gamma-\beta } \rk)e ^{-\sigma T} \sum_{N=1} ^{n^\gamma} {(\sigma T) ^N \over N!
}\to 1
\EEQS
as $n\to \infty$. Thus $\lim_{n\to\infty} \PP^L\lk( B_n\rk)=1$ and the proof is complete.

\del{
It remains to show (i) and (ii).
Given the jump times  $\{T_1,\cdots, T_N\}$ and large jumps $\{Y_1,\ldots,Y_n\}$,
we know from Step I that $u\in \mathfrak{X}_\lambda$.
In particular,
$$
\EE \int_0^ T |u(s)|^ q\, ds <\infty.
$$
In addition, $u$ solves %Step I we have seen that
\begin{align} %\label{largejumps}
\\
\nonumber
\lqq{
u(t)=S(t)u_0+\int_{0}^{t}S(t-s)F(u(s))\,ds}
&&
\\
\nonumber
&&{} +\int_{0}^{t}\int_Z S(t-s)G(s,u(s),z)\,\tilde \eta (dz,ds)
\\
\nonumber
&&{} +
\sum_{i=1}^ N 1_{[T_i,T]} (t)\CT(t-T_i)G_L(T^ -_i,u(T ^-_i),Y_i). %, \quad \textrm{a.s.}
\end{align}
Taking expectation leads to
\begin{align} %\label{largejumps}
\\
\nonumber
\lqq{
\EE\int_0^ T |u(t)|^ m\, dt =\EE\int_0^ T | (\mT u_0)(t) |^m\, dt +\EE \int_0^ T \lk|\lk(\mathfrak{F}\xi\rk) (t)\rk|^m}
&&
\\
\nonumber
&&{} +\EE\int_{0}^{T}\lk| \lk(\mathfrak{G}\xi\rk)(t) \rk|^m \, dt
 +
\int_0^ T \EE \lk|\sum_{i=1}^ N 1_{[T_i,T]} (t)\CT(t-T_i)G_L(T^ -_i,u(T ^-_i),Y_i)\rk|^m\, dt . %, \quad \textrm{a.s.}
\end{align}

\DEQSZ\label{mZlarge}
\\
\nonumber
\lqq{ (\mathfrak{Z}_0(\xi;(x_i)_{i=1}^ N)(t):= (\mT u_0)(t)+
\lk(\mathfrak{F}\xi\rk) (t)}
\\ &&{}\nonumber+
\lk( \mathfrak{G}\xi\rk)(t) + \sum_{i=1}^ N 1_{[T_i,T]}(t) (\mT_{T_i}  G_L(T^-_i,x_i,Y_i))(t)
, \quad t\in[0,T],\, \xi\in\mX_\lambda.
  \EEQSZ
}
%

%Each of the factor converges to one, therefore the product converges to one as well.
%Since
%%
%$$ \lk\{ \mbox{there exists solution to \eqref{SDE-bigjumps}}
%\rk\} \supset \cup_{n\in\NN} B_n ,
%$$
%and $B_n\supset B_{n+1}$ we have
%$$
%\PP\lk( \lk\{ \mbox{there exists solution to \eqref{SDE-bigjumps}}
%\rk\}\rk) \ge \lim_{n\to\infty} \PP\lk( B_n\rk)=1.
%$$
%we have shown that $\PP$--a.s.\ there exists a solution to \eqref{SDE-bigjumps}.
}
\end{proof}

\section{Application to isotropic synchronous viscoelactic materials} \label{reaction1-s-t}
%surface flow}
One motivation for this work are the linear models of viscoelasticity. In the linear theory for homogeneous isotropic viscoelastic the velocity fileld $v$ of the material occupying a bounded region $\mathcal{O}\subset \mathbb{R}^d$, $d=1,2,3$, with $C^1$ boundary, which rests up to $t=0$, is governed by
\begin{equation}\label{eq:viscoel}
\begin{aligned}
\dot{v}(x,t)&= \int_0^t \Delta v(x,t-s)\, da(s)\\
&+\int_0^t \nabla \nabla \cdot v(x,t-s)\, d(c(s)+\frac{1}{3}a(s))+g(x,t),\quad x\in \mathcal{O},
\end{aligned}
\end{equation}
where $g$ denotes an external body force, see, for example, \cite[Chapter 5]{pruss}. For simplicity, we set the density of the material equal to constant $1$. The functions $a$ and $c$ are called the shear modulus and the compression modulus, respectively. The material is called synchronous if for some $\gamma >0$ it holds that $c(t)=\gamma a(t)$. Suppose that $a(t)=\int_0^tb(s)\, ds$. In this case \eqref{eq:viscoel} simplifies to
\begin{equation}\label{eq:viscoel1}
\begin{aligned}
\dot{v}(x,t)&= \int_0^t \Delta v(x,t-s)\, b(s)ds\\
&+\int_0^t \nabla \nabla \cdot v(x,t-s)\, (\gamma+\frac{1}{3})b(s)ds+g(x,t),\quad x\in \mathcal{O},
\end{aligned}
\end{equation}
We consider the external body force to be abrupt in irregular time and space instances % but with a given deterministic direction $0\neq y_0\in \mathbb{R}^d$
and it is modelled by  $g(x,t)=\dot{L}(x,t)$ where $L$ is a space time  $\beta$--stable L\'evy process with $\beta<1$, %  with intensity measure $\nu$,
having only positive jumps. For the definition of the space time L\'evy noise, we refer to \cite[Section A.3, Definition A.13]{zpe}.
 We supplement \eqref{eq:viscoel1} by initial and Dirichlet zero boundary conditions:
\begin{equation}\label{eq:bc}
v\big|_{\partial \mathcal{O}}=0,\quad v(0,x)=v_0(x).
\end{equation}
Now, applying Theorem \ref{global} to \eqref{eq:viscoel1} with boundary conditions \eqref{eq:bc} with $g(x,t)=\dot{L}(t,x)y_0$, following result can be stated. For simplicity, we take $v_0=0$.
\begin{corollary}
Under the assumption above, if the kernel $b$ satisfies Assumptions \ref{ass2} with $\rho<\frac 4{d}$, then equation \eqref{eq:viscoel1} has a unique mild solution $u$ such that
for any $\alpha>\frac{d}{4}$, $\PP$-a.s.,\ $u\in\DD([0,T];H^{-2\alpha}_2(\CO))$.
\end{corollary}
\begin{proof}
{\rm
%We will show that the assumptions of Theorem \ref{global} are satisfied by using  Proposition \ref{prop-delta}.
% thatone may embed the space time \levy noise in a Besov space $B_{2,\infty}^ {\frac d2 -d}(\CO)$.
First, we split the noise into two noises, one with small, but bounded jumps, the other consisting of the large jumps.
%Taking $Z=\CO\times (0,b]$ and $Z_L=\CO \times (b,\infty)$ where , we can apply Theorem \ref{global}
Let $r>0$ be an arbitrary number,  $\eta$ be a space time Poisson random measure on $\CO\times \RR$  with jump intensity
$$\nu^\RR(U) := \int_{U\cap [0,r]} y^{-\beta-1}\, dy, \quad U\in\CB(\RR)
$$
and $\eta^\RR_L$ be a space time Poisson random measure\footnote{For the Definition of space time Poisson random measure we refer to  \cite[Section A.3, Definition A.16]{zpe} .} on $\CO\times \RR$  with jump intensity
$$\nu^\RR_L(U) := \int_{U\cap (r,\infty)} y^{-\beta-1}\, dy, \quad U\in\CB(\RR).
$$

First we verify the assumptions of Theorem \ref{local_ex}.
Let us put $H=L^2(\CO)$ and let  $m=\mbox{Leb}\times\nu^ \RR $ be a measure  on $\CO\times \RR$.
Consider the function
\begin{equation}\label{phidef}
\Phi:(x,y)\in\CO\times \RR \mapsto \Phi(x,y) := (\delta_x)\, y
\end{equation}
 It is not hard to see that $\Phi :\CO\times \RR \mapsto B_{2,\infty}^ {-\frac d2}(\CO)$ is measurable.
For the definition of the Besov spaces $B_{2,\infty}^ {-s}(\mathbb{R}^d)$ and $B_{2,\infty}^ {-s}(\CO)$ see, for example,  \cite[Appendix C]{zpe}.
Indeed, it follows from a straightforward calculation, see also, \cite[(C5)]{zpe}, that
\DEQSZ\label{zuinter}
\sup_{x\in \mathbb{R}^d}|\delta_x|_{B_{2,\infty}^ {-\frac d2}(\mathbb{R}^d)}<\infty.
\EEQSZ
\del{To show this, et us recall the
definition of the Besov spaces  as  given in \cite[Definition 2,
pp. 7-8]{Runst+Sickel_1996}. First we choose a function
$\psi\in{\mathcal{S}}(\mathbb{R}^d)$ such that $0\le \psi(x)\le 1$, $x\in
\mathbb{R}^d$ and
$$
\psi(x) = \left\{ \begin{array}{rcl} 1,&\mbox{ if } & |x|\le 1,\\
0&\mbox{ if } & |x|\ge \frac 32.
\end{array}\right.
$$
Then put \begin{eqnarray*}
 \left\{ \begin{array}{rcl}\phi_0(x) &=&\psi(x), \; x\in \mathbb{R}^d,
\\
\phi_1(x) &=&\psi(\frac x2)-\psi(x), \; x\in \mathbb{R}^d,
\\
\phi_j(x) &=&\phi_1(2 ^{-j+1} x),\; x\in \mathbb{R}^d, \quad
j=2,3,\ldots.
\end{array} \right.
\end{eqnarray*}
\begin{definition}
Let $s\in{\mathbb{R}}$,  $0<p\le \infty$ and $f \in {\mathcal{S}}^\prime(\mathbb{R}^d)$. If $0<q <
\infty$  we put
\begin{eqnarray*} \left| f\right|_{B^{s}_{p,q}} &=& \left( \sum_{ j=0} ^\infty  2
^{sjq}\left| {\mathcal{F}} ^{-1} \left[ \phi_j{\mathcal{F}} f\right] \right|_{L ^p} ^q\right)
^\frac 1q = \Vert\Big( 2 ^{sj}\left| {\mathcal{F}}
^{-1} \left[ \phi_j{\mathcal{F}} f\right]\right|_{L ^p}\Big)_{j\in \mathbb{N}} \Vert_{l^q}.
\end{eqnarray*}
If  $q=\infty$ we put
\begin{eqnarray*}
\left| f\right|_{B^{s}_{p,\infty}} &=& \sup_{j\in{\mathbb{N}}}  2 ^{sj}\left| {\mathcal{F}}
^{-1} \left[ \phi_j{\mathcal{F}} f\right]\right|_{L ^p}=\Vert \Big( 2 ^{sj}\left| {\mathcal{F}}
^{-1} \left[ \phi_j{\mathcal{F}} f\right]\right|_{L ^p}\Big)_{j\in \mathbb{N}} \Vert_{l^\infty} . \end{eqnarray*} We denote by
$B^{s}_{p,q}(\mathbb{R}^d)$ the space of all $f \in
{\mathcal{S}}^\prime(\mathbb{R}^d)$ for which $\left| f\right|_{B^{s}_{p,q}}$ is
finite.
\end{definition}
}
\del{First, note that for a sequence $a_n\to a$ we have
\begin{eqnarray*} 
|  {\mathcal{F}}^{-1}[\phi_j{\mathcal{F}}(\delta_{a_n}-\delta_a)|_{L ^p(\mathbb{R}^d)} &=&
(2\pi)^{-d/2}|  ({\mathcal{F}}^{-1}\phi_j)\ast (\delta_{a_n}-\delta_a)|_{L
^p(\mathbb{R}^d)} \\
\le C\,(2\pi)^{-d/2} |{\mathcal{F}}^{-1}\phi_j|_{L ^p(\mathbb{R}^d)}
&=&(2\pi)^{-d/2}2^{d(\frac1p-1)}2^{-jd(\frac1p-1)}  |{\mathcal{F}}^{-1}\phi_1|_{L
^p(\mathbb{R}^d)}.
 \end{eqnarray*}
Secondly, n}
Note that  $| \CF (\delta_a-\delta_b)(\xi)|\le c\, |\xi|\,|a-b|$, where $\CF$ denotes the Fourier transform, and  therefore, a short calculation yields that
$$
\|\delta_a-\delta_b\|_{B_{2,\infty}^ {-\frac d2-1}(\mathbb{R}^d)}\le C |a-b|.
$$
%\begin{eqnarray*}
%\lqq{ |  {\mathcal{F}}^{-1}[\phi_j{\mathcal{F}}(\delta_{a_n}-\delta_a)|_{L ^p(\mathbb{R}^d)} =
%(2\pi)^{-d/2}|  ({\mathcal{F}}^{-1}\phi_j)\ast (\delta_{a_n}-\delta_a)|_{L
%^p(\mathbb{R}^d)}
%} &&
% \\
%&\le &C\,(2\pi)^{-d/2} 2 ^ {j} |{\mathcal{F}}^{-1}\phi_j|_{L ^p(\mathbb{R}^d)}|a_n-a|
%\\
%&=&(2\pi)^{-d/2}2^{d(\frac1p-1)}2^{-jd(\frac1p-1)} 2 ^ {j}  |{\mathcal{F}}^{-1}\phi_1|_{L
%^p(\mathbb{R}^d)}|a_n-a|\le C\, 2^ j |a_n-a|.
% \end{eqnarray*}
% Since   ${\mathcal{F}}^{-1}(\vp
%u)=(2\pi)^{-d/2}({\mathcal{F}}^{-1}\vp)\ast ({\mathcal{F}}^{-1} u)$, $\vp \in {\mathcal{S}}$, $u
%\in {\mathcal{S}}^\prime$ we infer that for $j\in \mathbb{N}^\ast$,
%
%
It follows from above by interpolation using in addition \eqref{zuinter},  that for any $s>\frac d2$, $a\to \delta_{a}$ is continuous in $B_{2,\infty}^ {-s}(\mathbb{R}^d)$. 
%The continuity of the map $ \RR\ni c \mapsto c\delta_a\in B_{p,\infty}^ {-s}(\mathbb{R}^d)$ 
%is obvious and hence 
The continuity of the product $(x,y)\in \mathbb{R}^d\times \mathbb{R}\to (\delta_x)y$ in $B_{2,\infty}^ {-s}(\mathbb{R}^d)$ now easily follows. Finally if $a,b\in \mathcal{O}$, then the support of $y_1\delta_a-y_2\delta_b$ is in $\mathcal{O}$ and hence, using the definition of $B_{2,\infty}^ {-s}(\CO)$, see, for example \cite[Definition 4.2.1]{Triebel_1995}, it follows that
$$
\|y_1\delta_a-y_2\delta_b\|_{B_{2,\infty}^ {-s}(\CO)}\leq \|y_1\delta_a-y_2\delta_b\|_{B_{2,\infty}^ {-s}(\mathbb{R}^d)}
$$
and thus $\Phi$ defined in \eqref{phidef} is continuous from $\CO\times \RR$  to $B_{2,\infty}^ {-s}(\CO)$ hence measurable.
Therefore, $\Phi$ induces a (L\'evy) measure $\nu$ on the space $B^ {-s}_{2,\infty}(\CO)$ (push forward measure) by setting
$$
 \nu (U) := \int_{\CO\times \RR} 1_{U}(\Phi (x,y))\,dx\, \nu^\RR(dy),\quad U\in \mathcal{B}( B^ {-s}_{2,\infty}(\CO)).
$$

%First we verify the assumptions of Theorem \ref{local_ex}.
%Let us put $H=L^2(\CO)$ and let  $m=\mbox{Leb}\times\nu^ \RR $ be a measure  on $\CO\times \RR$.
%Similarly to \cite[Section 6, (6.8)]{zpe}  we define a map
%\DEQSZ\label{phidef}
%\Phi:L^ 1(\CO)\to L^1(\CO\times \RR;m %\nu^\RR
%;B_{2,\infty}^ {-(d-\frac d2)}(\CO))
%\EEQSZ
%by
%$$
%[\Phi v](x,y) = (\delta_x)\, y,\quad (x,y)\in\CO\times \RR.
%$$
%For the definition of the Besov spaces $B_{2,\infty}^ {-(d-\frac d2)}(\CO)$ see, for example,  \cite[Appendix C]{zpe}.
%The boundedness of the map $\Phi$ follows by  Proposition C.4 along with Corollary C.4 in \cite{zpe}.
%Mimicking the calculations in \cite{zpe} one can easily show that the mapping $\Phi$ is linear and
%$$
%\int_{\CO\times \RR} | [\Phi v](x,y)|_{B^ {-(d-\frac d2)}_{2,\infty}(\CO)}\,dx \,\nu^\RR(dy)\le C\, |v|_{L^ 1(\CO)},
%$$
%hence, measurable.
%It is straightforward to see, that
%$\Phi$ induces a L\'evy measure $\nu$ on the space $B^ {-(d-\frac d2)}_{2,\infty}(\CO)$ by setting
%$$
% \nu (U) := \int_{\CO\times \RR} 1_{U}(\Phi (x,y))\,dx\, \nu^\RR(dy),\quad U\in \mathcal{B}( B^ {-(d-\frac d2)}_{2,\infty}(\CO)).
%$$

%Let $\delta_G\in[0,\frac 1{\rho p})$ and $\alpha_G=\delta_G/2$.
 Secondly, by  Sobolev embedding theorems (see \cite[p.\ 29, Chapter 2.2]{runst})
it follows that for any $\theta>\frac d2 $ there exists a $\theta_1>\frac d2$ such that we have
\DEQSZ \label{oooben}
B_{2,\infty}^ {-\theta_1}(\CO)\hookrightarrow H_2^{- \theta}(\CO)
\EEQSZ
continuously. In addition,  since
$
D(A)=H^ 2 _{2,0}(\CO)
$
we have $H^ {-\theta}_2(\CO)=H^A_{-\frac \theta 2}$.
Let $\theta$ and $\theta_1>\frac d2$ be chosen such that \eqref{oooben} holds.
Now, since the embedding from $B^ {-\frac d2}_{2,\infty}(\CO)$ into $H^ {-\theta}_2(\CO)$ is continuous, and $H^ {-\theta}_2(\CO)=H^A_{-\frac \theta 2}$,
for $Z= B^{-\theta_1}_{2,\infty}(\CO)$ the mapping
\DEQSZ\label{defG}
G: [0,T]\times H\times Z  &\longrightarrow &H^A_{-\frac \theta 2} % H_{-\delta_G}^ A
\\\nonumber
(t,v,w) & \mapsto &   \,w,
\EEQSZ
is  Lipschitz continuous. Let us put $\alpha_G=\frac \theta 2>\frac d4$.
The mapping $G$ satisfies   Assumption
 \ref{hypogeneral}-(i), if
there exists some $q>p$ such that $\alpha_G< \frac 1 {q\rho}$.
Since $\nu^ \RR$ has bounded moments of all order and $\rho<\frac 4{d}$, such a $q$ exists.

As we have only positive jumps, we have to add and subtract the compensator. Hence, defining a mapping
\DEQSZ\label{defFF}
F: [0,T]\times H&\longrightarrow & H_{-\alpha_F}^ A
%\\
%(t,v,a) & \mapsto &   \delta_a\, x_\nu,
\EEQSZ
by $F(t,v)(x)=\delta_x x_\nu$, where $x_\nu=\int_{(0,b]} y^ {-\beta}\, dy$.
Setting $\alpha_I=\alpha_G$, we know $\alpha_I=\alpha_F$ and therefore, condition \eqref{addasszweites} is trivially satisfied.
Again, it is straightforward
to show that for $\alpha_F=\alpha _G$ the mapping $F$ is  Lipschitz continuous and  Assumption \ref{hypogeneral} (ii) is satisfied.

As $v_0=0$, Assumption \ref{hypogeneral} (iii) is fulfilled with $\alpha_I=\alpha_G$.

Next, we have to show that Assumption \ref{assbigjumps} holds for bigger jumps. Let us put $Z_L= B^{-\theta_1}_{2,\infty}(\CO)$
and $m_L=\mbox{Leb}\times\nu_L^ \RR $ be a measure on $\CO\times \RR$.
By arguing as before, one can show that $\Phi$ defined in \eqref{phidef}
induces a L\'evy measure $\nu_L$ on the space $B^ {-\theta_1}_{2,\infty}(\CO)$ by setting
$$
 \nu_L (U) := \int_{\CO\times \RR} 1_{U}(\Phi (x,y))\, \nu^\RR_L(dy)\, dx,\quad U\in \mathcal{B}( B^ {-\theta_1}_{2,\infty}(\CO)).
$$
 Then,  Assumption
\ref{slowvarbigjumps} is satisfied with $\beta$ being the index of the L\'evy process.

The mapping
\DEQSZ\label{resf}
G_L: [0,T]\times H\times Z_L  &\longrightarrow &  H_2^{- \theta}(\CO)
\\
(t,v,z) & \mapsto & \, z,\nonumber
\EEQSZ
clearly satisfies Assumption \ref{assbigjumps}.
 %Setting $\alpha_I=\alpha_G$, we know $\alpha_I>\frac 12$ and therefore, condition \eqref{addasszweites} is trivially satisfied.
%
%
%
Let us denote by $\eta$ the space time Poisson random measure with intensity $\nu$, $\tilde \eta$ the compensated  space time Poisson random measure and
by $\eta_L$ the  space time Poisson random measure with intensity $\eta_L$.
In this way, equation \eqref{eq:viscoel1} can be written in the following form
 \begin{equation}
 \left\{ \begin{aligned} du(t) & = \lk( A\int_0 ^t b(t-s) u(s)\,ds\rk) \, dt  + F(t,u(t))\,dt \\
& {} + \int_Z G(t,u(t)z)\,\tilde \eta(dz,dt)+
\int_{Z_L} G_L(t,u(t),z)\, \eta_L(dz,dt);\,  t\in (0,T],\\
u(0)&=u_0, \end{aligned} \right.
\end{equation}
where $G$, $G_L$ and $F$ are satisfying the assumptions of
 Theorem \ref{global}. Hence, we get  existence and uniqueness of a mild solution with \cadlag paths in $H_{2}^ {-2\alpha}(\CO)$
%For $p$ we take a number closed to $1$. Since $\rho<2$ we can find such numbers $p$ and $q$.
for  $\alpha>\frac{d}{4}$. \del{is chosen such that
\DEQSZ\label{musssat}
 \left((\alpha_G-\alpha)\rho+1\right)  \frac {pq}{q-p}< 1
\EEQSZ
As $p=1$ and  $q$ can be taken arbitrary large $\alpha$ has to satisfy $\left(\alpha_G-\alpha)\rho+1\right)<1$; that is,
$\alpha>\alpha_G>\frac{d}{4}$. In particular, condition \eqref{musssat} holds and for any $\alpha>\frac d4$ and
the solution process $u$ is \cadlag in $H_{2}^ {-2\alpha}(\CO)$.}
}
\end{proof}

\end{document}